\documentclass[10pt, a4paper]{amsart}

\usepackage{charter}
\usepackage{bm}
\usepackage{amsmath}
\usepackage{amssymb}
\usepackage{amsthm}
\usepackage{graphicx}
\usepackage{url}
\usepackage{enumerate}
\usepackage{xcolor}
\usepackage{mathrsfs} 
\usepackage{commath}
\usepackage{tikz}
\usetikzlibrary{automata}

\usepackage{t1enc}
\usepackage{dsfont}

\usepackage[margin=3cm]{geometry}
\newtheorem{thm}{Theorem}

\newtheorem{lem}[thm]{Lemma}
\newtheorem{prop}[thm]{Proposition}

\newtheorem{cor}[thm]{Corollary}

\theoremstyle{definition}
\newtheorem{defn}[thm]{Definition}
\newtheorem{rem}[thm]{Remark}
\newtheorem{exmp}[thm]{Example}

\DeclareMathOperator{\freq}{freq}

\renewcommand{\phi}{\varphi}

\DeclareMathOperator{\M}{\mathcal{M}}

\DeclareMathOperator{\Ms}{\mathcal{M}_{\mathit \sigma} }

\DeclareMathOperator{\Mse}{\M_{\mathit \sigma}^e}

\newcommand{\FS}{\alf^\infty}

\newcommand{\BB}{\mathscr{B}}

\newcommand{\Bl}{\mathcal{B}}

\newcommand{\alf}{\mathscr{A}}

\DeclareMathOperator{\lang}{\Bl}

\DeclareMathOperator{\CL}{CL}
\DeclareMathOperator{\dirac}{\delta}

\newcommand{\dbar}{{\bar d}}
\newcommand{\dbarm}{\bar{d}_{\mathcal{M}}}

\newcommand{\Hdbar}{{\bar d}^H}
\newcommand{\Hdbarm}{\bar{d}^H_{\mathcal{M}}}

\newcommand{\dHam}{d_\textrm{Ham}}

\newcommand{\weakst}{weak$^*$}

\newcommand{\eps}{\varepsilon}
\newcommand{\ajoining}{\xi}
\newcommand{\emptyword}{\lambda}

\newcommand{\R}{\mathbb{R}}
\newcommand{\Z}{\mathbb{Z}}
\newcommand{\N}{\mathbb{N}}

\newcommand{\betn}{\beta^{(n)}}

\newcommand{\cB}{\mathcal{B}}

\author{Melih Emin Can \and Jakub Konieczny \and Michal Kupsa \and Dominik Kwietniak}

\address[M. E.\ Can]{Faculty of Mathematics and Computer Science, Jagiellonian University in Krakow, ul. \L o\-jasiewicza 6, 30-348 Krak\'ow, Poland}\email{melih.can@im.uj.edu.pl}

\address[D.\ Kwietniak]{Faculty of Mathematics and Computer Science, Jagiellonian University in Krakow, ul. \L o\-jasiewicza 6, 30-348 Krak\'ow, Poland}\email{dominik.kwietniak@uj.edu.pl}
\urladdr{www.im.uj.edu.pl/DominikKwietniak/}

\address[M.\ Kupsa]{ The Czech Academy of Sciences, Institute of Information Theory and Automation, Prague 8, CZ-18208}\email{kupsa@utia.cas.cz}

\address[J.\ Konieczny]
{Department of Computer Science, University of Oxford,
Wolfson Building, Parks Road, Oxford OX1 3QD, UK}
\email{jakub.konieczny@gmail.com}

\title[Minimal and proximal examples\ldots]{Minimal and proximal examples of $\dbar$-stable and $\dbar$-approachable shift spaces}
\date{\today}

\begin{document}

\begin{abstract}
We study shift spaces over a finite alphabet that can be approximated by mixing shifts of finite type in the sense of (pseudo)metrics connected to Ornstein's $\dbar$ metric ($\dbar$-approachable shift spaces). The class of $\dbar$-approachable shifts can be considered as a topological analog of measure-theoretical Bernoulli systems. The notion of $\dbar$-approachability together with a closely connected notion of $\dbar$-shadowing were introduced by Konieczny, Kupsa, and Kwietniak [in \emph{Ergodic Theory and Dynamical Systems}, vol.	\textbf{43} (2023), issue 3, pp. 943--970].  These notions were developed with the aim to significantly generalize specification properties. Indeed, many popular variants of the specification property, including the classic one  and almost/weak specification property ensure $\dbar$-approachability and $\dbar$-shadowing.
Here, we study further properties and connections between $\dbar$-shadowing and $\dbar$-approachability. We prove that $\dbar$-shadowing implies $\dbar$-stability (a notion recently introduced by Tim Austin).
We show that for surjective shift spaces with the $\dbar$-shadowing property the Hausdorff pseudodistance $\Hdbar$ between shift spaces induced by $\dbar$ is the same as the Hausdorff distance between their simiplices of invariant measures with respect to the Hausdorff distance induced by the Ornstein's metric $\dbar$ between measures. We prove that without $\dbar$-shadowing this need not to be true (it is known that the former distance always bounds the latter). We provide  examples illustrating these results including minimal examples and proximal examples of  shift spaces with the $\dbar$-shadowing property. The existence of such shift spaces was announced in our earlier paper [op. cit.]. It shows that $\dbar$-shadowing indeed generalises the specification property.
\end{abstract}
\maketitle

Given a finite set (an \emph{alphabet}) $\alf$ we let $\FS$  stand for the \emph{full shift over $\alf$}, that is a set of all $\alf$-valued infinite sequences. To avoid trivialities, we assume that $\alf$ has at least two elements. We endow $\FS$ with the product topology induced by the discrete topology on $\alf$, which turns $\FS$ into a compact metrisable space. Let $\rho$ be a metric compatible with the topology on $\FS$.  The shift operator
$\sigma\colon\FS\to\FS$ turns $\FS$ into a noninvertible dynamical system. From the dynamical point of view, the most interesting objects are closed nonempty $\sigma$-invariant subsets of $\FS$ (\emph{one-sided shift spaces} or \emph{subshifts}). We also consider the space $\M(\FS)$ of all Borel probability measures on $\FS$ with the weak$^*$ topology. The set of $\sigma$-invariant measures in $\M(\FS)$ concentrated on a shift space $X\subseteq\FS$ is denoted by $\Ms(X)$.
Each of these objects (invariant measures and subshifts) has a canonically defined sequence of  Markov approximations converging to it in a natural topology.
This fact, however, is of little practical use, because the convergence is too weak to allow for a transfer of dynamical properties from an approximating sequence to the properties of its limit.

Recall that the natural topology on the space of all subshifts of $\FS$ is the hyperspace (Vietoris) topology of nonempty closed subsets of a compact metric space. In other words, a sequence of shift spaces $(X_n)_{n=1}^\infty\subseteq\FS$ converges to a shift space $X\subseteq \FS$ in the hyperspace topology if $\rho^H(X_n,X)\to 0$ as $n\to\infty$ (here, $\rho^H$ is the Hausdorff metric corresponding to $\rho$). Similarly, we say that simplices of invariant measures of shift spaces $(X_n)_{n=1}^\infty\subseteq\FS$ approximate the simplex of invariant measures of a shift space $X\subseteq \FS$ if $\Ms(X_n)$ converges to $\Ms(X)$ as $n\to \infty$  in the natural hyperspace topology of $\Ms(\FS)$, that is, if
$D^H(\Ms(X_n),\Ms(X))\to 0$ as $n\to\infty$, where $D^H$ is the Hausdorff metric corresponding to a metric $D$ compatible with the weak$^*$ topology on $\Ms(\FS)$.

Fortunately, for both measures and subshifts, stronger metrics than $\rho$ and $D$ are also available. A useful metric  for $\sigma$-invariant measures is Ornstein's  metric $\dbarm$. In \cite{KKK2} we studied a topology on the powerset of $\FS$ induced by the Hausdorff pseudometric $\Hdbar$ derived from $\dbar$-pseudometric on $\FS$. A very similar idea of using $\dbar$-approximation was independently considered by Dan Thompson \cite{Thompson-note}, who used it in the settings of \cite{CT}.

Recall that the pseudometric $\dbar$ is given for $x=(x_j)_{j=0}^\infty,\ y=(y_j)_{j=0}^\infty\in \FS$ by
\begin{equation}\label{def:dbar}
\dbar(x,y)=\limsup_{n\to\infty}\frac{1}{n}|\{0\le j <n : x_j\neq y_j\}|.
\end{equation}
Since $\dbar(x,y)$ can be zero for distinct $x$ and $y$, $\dbar$ is not a metric. Nevertheless, after factorizing by the equivalence relation $\sim$ on $\FS$, where $x\sim y$ if and only if $\dbar(x,y)=0$, we obtain the factor space $\FS/{\sim}$ on which $\dbar$ becomes a complete, non-separable metric. Since $\dbar$ is bounded by $1$ on $\FS$, it induces a Hausdorff pseudometric $\Hdbar$ on the space $\CL(\FS,\dbar)$ of all nonempty $\dbar$-closed subsets of $\FS$.
Similarly, $\dbarm$ is a complete bounded non-separable metric on $\Ms(\FS)$ inducing a Hausdorff metric $\Hdbarm$ on the space $\CL(\Ms(\FS),\dbarm)$ of all nonempty $\dbarm$-closed subsets of $\Ms(\FS)$. Note that $\dbarm$-convergence implies weak$^*$ convergence, so for each shift space $X\subseteq\FS$ the set $\Ms(X)$ is $\dbarm$-closed. It is also known that the set of ergodic measures on $X$, denoted by $\Mse(X)$, is $\dbarm$-closed.

Hence, we obtain two more ways to say that shift spaces $(X_n)_{n=1}^\infty\subseteq\FS$ approximate $X\subseteq \FS$: 
\begin{align}
  \lim_{n\to\infty}&\Hdbar(X_n,X)=0. \label{four-modes-4}\\
\lim_{n\to\infty}&\Hdbarm(\Ms(X_n),\Ms(X))=0,\label{four-modes-1}
\end{align}
Note that we do not assume that the approximation in \eqref{four-modes-4} and \eqref{four-modes-1} is monotone (meaning $X_1\supseteq X_2 \supseteq \ldots$ and $X=\bigcap X_n$), but in practice it is often the case.

In \cite{KKK2}, we studied the consequences of the existence of an approximating sequence as in \eqref{four-modes-4} and \eqref{four-modes-1}. We were especially interested in the case when the approximating sequence is the sequence of Markov approximations. We introduced \emph{$\dbar$-approachable} shift spaces (subshifts that are approached by their topological Markov approximations not only in the `usual' Hausdorff metric topology, but also in
the $\Hdbar$ sense). We also considered a condition that is ostensibly a relaxation of \eqref{four-modes-1}:
\begin{equation}
\lim_{n\to\infty}\Hdbarm(\Mse(X_n),\Mse(X))=0. \label{four-modes-2}
\end{equation}
We proved in \cite{KKK2} that for every shift spaces $X$ and $Y$ over $\alf$ we have
\begin{equation}\label{hdbar-ineq}
\Hdbarm(\Ms(X),\Ms(Y))=\Hdbarm(\Mse(X),\Mse(Y))\le \Hdbar(X,Y),
\end{equation}
hence
\[
\eqref{four-modes-4}
\implies \eqref{four-modes-1}\Longleftrightarrow\eqref{four-modes-2}.
\]
In other words, $\Hdbar$ approximation \eqref{four-modes-4} implies convergence of simplices of invariant measures in the Hausdorff metric $\Hdbarm$ induced by Ornstein's $\dbarm$ metric on the space $\Ms(\FS)$ as in \eqref{four-modes-1}. As a consequence, certain features of simplices of invariant measures of shift spaces in the approximating sequence are inherited by the simplex of the limit.

In analogy\footnote{Note that Ornstein's theory characterising Bernoullicity works best for invertible measure-preserving systems, in particular Markov shifts carrying Markov approximations mentioned here, are two-sided, while we consider one-sided setting.} with Friedman and Ornstein's result characterizing Bernoulli measures among all totally ergodic shift invariant measures as $\dbarm$-limits of their own Markov approximations (see \cite{FO70}), we also characterized in \cite{KKK2} chain mixing $\dbar$-approachable shift spaces using the newly introduced $\dbar$-shadowing property. In this way we obtained a large family of shift spaces that contains all $\beta$-shifts and all mixing sofic shifts, in particular all mixing shifts of finite type. This is because many specification properties imply chain mixing and $\dbar$-approachability (this is the case, for example, for all shift spaces with the almost specification property). We refer to \cite{KKK2} for more details.

We also showed in \cite{KKK2} that if every $X_n$ has entropy-dense set of ergodic measures and the sequence $(X_n)_{n=1}^\infty$ converges to $X$ in the sense defined by any of \eqref{four-modes-4}--\eqref{four-modes-2}, then ergodic measures of $X$ are also entropy-dense. This established a new method of proving entropy density. As a consequence we obtained entropy density of ergodic measures for all surjective shift spaces with the $\dbar$-shadowing property. Entropy density of ergodic measures is a property introduced by Orey in 1986 \cite{Orey} and F\"ollmer and Orey in 1988 \cite{FO}. Recall that  ergodic measures of a shift space $X$ are entropy-dense if every invariant measure can be approximated with an ergodic one with respect to the \weakst topology and entropy at the same time.
In particular, $\Mse(X)$ is a dense subset of $\Ms(X)$. Note that there are shift spaces with dense, but not entropy-dense, sets of ergodic measures (see \cite{GK}). Density of ergodic measures and entropy density are strongly related to the theory of large deviations and multifractal analysis \cite{Comman,EKW,PS,PS2}, see Comman's article \cite{Comman2017} and references therein for more information about that connection.

The results of \cite{KKK2} are also applicable to the study of dynamics of the so-called $\mathscr{B}$-free shifts (or systems), a subject that has recently attracted a considerable interest (see \cite{DKPS,DKKPL, KKLem, Keller, Keller-Studia, KKK, KPL, KPLW, KPLW2}.

In the present paper, we study further properties of $\dbar$-approachable shift spaces. In particular, in Section \ref{sec:examples} we construct minimal and proximal examples of chain mixing $\dbar$-approachable shift spaces. 
These examples demonstrate that our technique yields entropy density for shift spaces that are beyond the reach of methods based on specification, as specification excludes both proximality and minimality. So far only specification-like conditions were invoked to prove entropy density explicitly (see \cite{EKW,PS}). We note that there exists a general theorem due to Downarowicz and Serafin \cite{DS} that guarantees existence of minimal shifts with entropy-dense ergodic measures, 
 but due to its generality it is hard to see concrete examples.
We also prove, see Section \ref{sec:dbar-stab}, that $\dbar$-approachability implies $\dbar$-stability, where $\dbar$-stability is a property recently introduced by Tim Austin \cite{Austin}. Austin combined one of Ornstein’s conditions equivalent to Bernoullicity with equivariant analogs of some basic results in measure concentration to characterize Bernoullicity of the equilibrium measure of a continuous potential $\varphi$ under the assumption that the equilibrium is
unique. Austin formulated his main condition in terms of a stronger kind of differentiability of the
pressure functional at $\varphi$. He proved that the condition is always necessary and he showed that it is sufficient if the shift space is ‘$\dbar$-stable’.
Austin remarked that the class of ‘$\dbar$-stable’ subshifts includes the full shift and several others examples with the specification property. He also suspected that $\dbar$-stability
holds also for examples without any specification properties. Hence our minimal and proximal examples of $\dbar$-approachable shifts confirm this suspicion.
In Section \ref{sec:Hdabr-vs-Hdbarm} we also show that the implication \eqref{four-modes-1}$\implies$\eqref{four-modes-4} holds true if all shift spaces involved have the $\dbar$-shadowing property, hence chain mixing and $\dbar$-approachability suffice for this implication, see Theorem \ref{thm:dbarshad-dbarm-dbar}. We note that the implication \eqref{four-modes-1}$\implies$\eqref{four-modes-4} does not hold in general by producing a sequence of shift spaces $(X_n)_{n=1}^\infty$ such that for some shift space $X$  we have
$\Hdbarm(\Ms(X_n),\Ms(X))\to 0$ while $\Hdbar(X_n,X)\to 1$ as $n\to\infty$, see Proposition \ref{prop:diff-limits}.

We prove (Proposition \ref{prop:no_simplex}) that it is possible to find a sequence of shift spaces $(X_n)_{n=1}^\infty$ such that $\Ms(X_n)$ converges in $\dbarm$ to a singleton set, that is not a simplex of invariant measures for any shift space. Finally, we prove that the $\dbar$-shadowing on the measure center of a shift space (the smallest invariant subshift of full measure for every invariant measure) implies the same for the shift. We recall our notation and basic definitions in Section \ref{sec:definitions}.

The results of the present paper, as well as the results of \cite{KKK2} are also applicable to two-sided shift spaces (shift-invariant subsets of $\alf^{\Z}$), provided that the definition of the pseudometric $\dbar$ stays as in \eqref{def:dbar}, viz. we average over the coordinates $0,1,\ldots,n-1$. 
\section{Definitions}\label{sec:definitions}


\subsection{Hausdorff pseudometrics}\label{subsec:Hausdorff}
Let $Z$ be a set. 
A \emph{pseudometric} on $Z$ is a real-valued, nonnegative, symmetric function $\rho$ on $Z\times Z$ vanishing on the diagonal $\{(x,y)\in Z\times Z: x=y\}$ and satisfying the triangle inequality. Let $\rho$ be a bounded pseudometric on $Z$. For $z\in Z$ and nonempty $A,B\subseteq Z$, we define
\[\rho(z,B)=\inf_{b\in B}\rho(z,b),\quad\text{and}\quad \rho^H(A,B)=\max\left\{\sup_{a\in A}\rho(a,B),\ \sup_{b\in B}\rho(b,A)\right\}.\]
We call $\rho^H$ the \emph{Hausdorff pseudometric} induced by $\rho$ on the space of all nonempty subsets of $Z$. 
If $\rho$ is a bounded metric, then $\rho^H$ becomes a metric on the set $\CL(Z,\rho)$ of closed nonempty subsets of $(Z,\rho)$.
Note that in our settings some properties, well-known in the compact case, fail, because we consider $(Z,\rho)$ where $\rho$ is not necessarily compact, but only a bounded pseudometric space. For example, $\rho$ and another pseudometric $\tilde\rho$ may induce the same topology on $Z$ but the spaces $(\CL(Z),\rho^H)$ and $(\CL(Z),\tilde{\rho}^H)$ need not be homeomorphic.

\subsection{Shift spaces and languages}
We let $\N$ denote the set of positive integers. We also write $\N_0=\N\cup\{0\}$. Unless otherwise stated, the letters $i,j,k,l,m,n$ always denote integers. An \emph{alphabet} is a finite set $\alf$ endowed with the discrete topology. We refer to elements of $\alf$ as to \emph{symbols} or \emph{letters}. The \emph{full shift} $\FS$  is the Cartesian product 
of infinitely many copies of $\alf$ indexed by $\N_0$. We endow $\FS$ with the product topology. A compatible metric on $\FS$ is given for $x,y\in\FS$ by
\[
\rho(x,y)=\begin{cases}
            0, & \mbox{if } x=y, \\
            2^{-\min\{j:x_j\neq y_j\}}, & \mbox{otherwise}.
          \end{cases}
\]
The \emph{shift map} $\sigma\colon\alf^\infty\to\alf^\infty$ is given for $x=(x_i)_{i=0}^\infty\in\FS$ and $j\ge 0$ by $\sigma(x)_j=x_{j+1}$.
A \emph{shift space} over $\alf$ is a nonempty, closed, and $\sigma$-invariant subset of $\FS$.
A \emph{word}  over $\alf$ is a finite sequence of elements of $\alf$. The number of entries of a word $w$ is called the \emph{length} of $w$ and is denoted by $|w|$.
The empty sequence is called the \emph{empty word} and is the only word of length $0$. We denote it by $\emptyword$.
The \emph{concatenation} of words $u=u_1\ldots u_k$ and $v=v_1\ldots v_m$ is the word $u_1\ldots u_k v_1\ldots v_m$ denoted simply as $uv$. Given $x \in \FS$ and $0\le i<j$ we let $x_{[i,j)}$ denote the word $x_ix_{i+1}\ldots x_{j-1}$ over $\alf$ of length $j-i$. 
We say that a word $w$ \emph{appears} in $x\in \FS$ if there exist  $0\le i<j$ such that
$w=x_{[i,j)}$. A word $w$ \emph{appears} in a shift space $X \subseteq \FS$ if there exists $x\in X$ such that $w$ appears in $x$. The \emph{language} of a shift space $X\subseteq\FS$ is the set $\lang(X)$ of all finite words over $\alf$ appearing in $X$. We agree that the empty word appears in every sequence in $\FS$. For $n\in\N_0$, we let $\lang_n(X) \subseteq \alf^n$ to be the set of all words $w\in\lang(X)$ with $|w|=n$. Given a set $\mathscr{F}$ of finite words over $\alf$ we define $X_{\mathscr{F}}$ to be the set of all $x=(x_i)_{i=0}^\infty\in\FS$ such that no word from $\mathscr{F}$ appears in $x$.
The resulting set $X_{\mathscr{F}}$ is either empty or a shift space. Furthermore, for every shift space $X$ over $\alf$ one can find a collection $\mathscr{F}$ of finite words such that $X=X_\mathscr{F}$. A shift space $X$ is a \emph{shift of finite type} if there exists a finite set $\mathscr{F}$ such that $X=X_{\mathscr{F}}$. Every shift space $X\subseteq\FS$ is the intersection of a sequence $(X^M_n)_{n\ge 0}$ of shifts of finite type. To construct that sequence, we define $\mathscr{F}[n]$ to consists of all words $w$ over $\alf$ with $|w|= n+1$ and $w\notin\lang_{n+1}(X)$.
In this way, we obtain for each $n\ge 0$ a shift of finite type $X^M_n=X_{\mathscr{F}[n]}$ such that $\lang_{j}(X)=\lang_j(X^M_n)$ for $0\le j\le n+1$. We call the shift space $X^M_n$ the $n$-th \emph{(topological) Markov approximation} of $X$ or \emph{finite type approximation of order $n$ to $X$}. We note that for every shift space $X$, its  Markov approximation $X^M_n$ can be conveniently described using a Rauzy graph.  The \emph{$n$-th Rauzy graph} of $X$ is a labeled graph $G_n=(V_n,E_n,\tau_n)$,
where we set $V_n=\lang_n(X)$ and $E_n=\lang_{n+1}(X)$, and
for each $w=w_0w_1\ldots w_{n}\in E_n$ we define $i(w)=w_0\ldots w_{n-1}\in V_n$,  $t(w)=w_1\ldots w_{n}\in V_n$, and $\tau_n(w)=w_0\in\alf$. The sofic shift space $X_n$ presented by $G_n$ satisfies $\lang_{j}(X_n)=\lang_{j}(X)$ for $j=1,\ldots,n+1$, see Proposition 3.62 in \cite{Kurka}. It is now easy to see that $X_n$ is the $n$-th  topological Markov approximations for $X$.

The following definitions of (chain) mixing and (chain) transitivity are stated only for shift spaces. We will do the same for several notions: instead of presenting a general definition for continuous maps acting on compact metric spaces (for the latter, see \cite{Kurka}), we will state an equivalent definitions adapted to symbolic dynamics. This applies to (chain) transitivity, (chain) mixing, specification and its variants.

A shift space $X$ is \emph{transitive} if for every $u,w\in\lang(X)$ there exists $v$ with $uvw\in\lang(X)$. A shift space $X$ is \emph{topologically mixing} if for any $u,w\in\lang(X)$ there exists $N\in\N_0$ such that for each $n\ge N$ there is  $v=v(n)\in\lang_n(X)$ such that $uvw\in\lang(X)$.

A shift space is \emph{chain transitive} (respectively, \emph{chain mixing})
if its topological Markov approximations $X^M_n$ are transitive (respectively, topologically mixing) for all except finitely many $n$'s.


\subsection{Ergodic properties of shift spaces}
Let $\M(X)$ be the set of all Borel probability measures supported on a shift space $X\subseteq\FS$. In particular, $\M(\FS)$ stands for the space of all Borel probability measures on $\FS$. We write $\Ms(X)$ and $\Mse(X)$ to denote the sets of, respectively, $\sigma$-invariant and ergodic $\sigma$-invariant measures in $\M(X)$. We endow $\M(\FS)$ with the \weakst\ topology, hence it becomes a compact metrisable space and $\Ms(X)$ is its closed subset for every shift space $X\subseteq\FS$.

We say that $x\in \FS$ \emph{generates}  $\mu\in\Ms(X)$ along a strictly increasing sequence of integers $(N_k)_{k=1}^\infty$,  if for every continuous function $f\colon X\to\R$ the sequence of Ces\`aro averages of $(f(\sigma^n(x)))_{n=0}^\infty$ along $(N_k)_{k=1}^\infty$ converges and the limit satisfies
\[
\lim_{k\to\infty}\frac{1}{N_k}\sum_{n=0}^{N_k-1}f(\sigma^n(x))=\int_{\FS}f\,\text{d}\mu.
\]
Compactness implies that for every strictly increasing sequence of integers $(N_k)_{k=1}^\infty$ and every point $x\in X$ there is a subsequence of $(N_k)_{k=1}^\infty$ such that $x$ generates an invariant measure along that subsequence.
In particular, every point always generates at least one measure 
A point  $x\in \FS$ is \emph{generic} for $\mu\in\Ms(X)$,  if $\mu$ is the unique measure generated by $x$. 
Every ergodic measure has a generic point.
By $h(\mu)$ we denote the Kolmogorov-Sinai entropy of $\mu\in\Ms(X)$.  We say that ergodic measures of a shift space $X$ are \emph{entropy dense} if
	for every measure $\mu\in \Ms(X)$, every neighborhood $U$ of $\mu$ in $\Ms(X)$ and every $\eps>0$ there is $\nu\in U\cap\Mse(X)$ with
	$|h(\nu)-h(\mu)|<\eps$.
Note that having entropy-dense ergodic measures is preserved by conjugacy. 

\subsection{The functions $\dbar$ and $\dbarm$}
Given $x=(x_n)_{n=0}^\infty,y=(y_n)_{n=0}^\infty\in\FS$ we set
\begin{gather*}
  \dbar(x,y)=\limsup_{n\to\infty}\frac{1}{n}|\{0\le j < n:x_j\neq y_j\}|. 
\end{gather*}
The function 
$\dbar$ is a pseudometric on $\FS$, but $\dbar$ is not a metric if $\alf$ has at least two elements, because the implication $\dbar(x,y)=0\implies x=y$ fails. The function $\dbar$ 
is not continuous in general.
Furthermore, $\dbar\colon\FS\times\FS\to [0,1]$ is shift invariant (for all $x,y\in\FS$ we have 
$\dbar(x,y)=\dbar(\sigma(x),\sigma(y))$).

Ornstein's metric $\dbarm$ on $\Ms(\FS)$ is usually\footnote{Ornstein's metric $\dbarm$ is usually denoted by $\dbar$, but in \cite{KKK2} as well as in this paper the distinction between $\dbar$ and $\dbarm$ is crucial.  We refer to $\dbarm$ as the `d-bar distance for measures' and we call $\dbar$ on  $\FS$ `pointwise d-bar' or simply `d-bar'.} defined with the help of joinings.
A $\sigma\times\sigma$-invariant measure $\ajoining$ on $\FS\times\FS$ is a \emph{joining} of $\mu,\nu\in\Ms(\FS)$ if $\mu$ and $\nu$ are the marginal measures for $\ajoining$ under the projection to the first, respectively the second, coordinate. We write $J(\mu,\nu)$ for the set of all joinings of $\mu$ and $\nu$. Note that $J(\mu,\nu)$ is always nonempty because the product measure $\mu\times\nu$ belongs to $J(\mu,\nu)$. Ornstein's metric $\dbarm$ on $\Ms(\FS)$ is given by:
\begin{align*}
  \dbarm(\mu,\nu)=\inf_{\ajoining\in J(\mu,\nu)}\int_{\FS\times\FS} d_0(x,y) \dif
   \ajoining(x,y),
\end{align*}
where $d_0(x,y)=1$ if $x_0\neq y_0$ and $d_0(x,y)=0$ otherwise.  The space $\Ms(\FS)$ endowed with $\dbarm$-metric becomes a complete but non-separable (hence, non-compact) metric space.
The space $\Mse(\FS)\subseteq\Ms(\FS)$ of ergodic measures is $\dbarm$-closed, as are the spaces of strongly mixing and Bernoulli measures on $\FS$.
Entropy function $\mu\mapsto h(\mu)$ is continuous on $\Ms(\FS)$ under $\dbarm$. The convergence in $\dbarm$ implies weak$^*$ convergence (for more details, see \cite{Ornstein}).

Applying Hausdorff metric construction described in Section \ref{subsec:Hausdorff} to the bounded metric space $(\Ms(\FS),\dbarm)$ we obtain a metric denoted by $\Hdbarm$ defined  on the space $\CL(\Ms(\FS),\dbarm)$ of nonempty closed subsets of $(\Ms(\FS),\dbarm)$. For every shift space $X$ the sets $\Ms(X)$ and $\Mse(X)$ are closed sets in the Hausdorff metric $\Hdbarm$. Similarly, starting from the pseudometric space  $(\FS,\dbar)$ 
we get a pseudometric $\Hdbar$ 
on the set of all nonempty subsets of $\FS$.

\subsection{
On $\dbar$-shadowing and $\dbar$-approachability}
A shift space $X \subseteq \FS$ is \emph{$\dbar$-ap\-proach\-able}
if its Markov approximations $X^M_1,X^M_2,\ldots$ satisfy $\Hdbar(X^M_n,X)\to 0$ as $n\to\infty$. Every shift of finite type is $\dbar$-approachable.

We say that shift space $X$ has the \emph{$\dbar$-shadowing property} if for every $\eps>0$ there is $N\in\N$ such that for every sequence $(w^{(j)})_{j=1}^\infty$
in $\lang(X)$ with $|w^{(j)}|\ge N$ for $j=1,2,\ldots$ is \emph{$\eps$-traced} by some point $x'\in X$, that is, there is $x'\in X$ such that $\dbar(x,x')<\eps$, where $x=w^{(1)}w^{(2)}w^{(3)}\ldots$. The $\dbar$-shadowing property was introduced in \cite{KKK2}. It is closely related to the \emph{average shadowing property} introduced by Blank \cite{Blank} and studied in \cite{KKO}. Every mixing sofic shift space has the $\dbar$-shadowing property and $\dbar$-shadowing is inherited by $\Hdbar$-limits of shift spaces with the $\dbar$-shadowing. Note that $\dbar$-shadowing implies $\dbar$-approachability, but to prove the converse we need to assume additionally that the shift space in question is chain mixing. The exact statement is Theorem 6 in \cite{KKK2}, which says that
a shift space $X\subseteq\FS$ is chain mixing and $\dbar$-approachable if and only if $\sigma(X)=X$ and $X$ has the $\dbar$-shadowing property. Theorem 6 in \cite{KKK2} lists a third condition, but we won't need it until Section \ref{sec:examples}, so we postpone the exact statement.

As we have already mentioned in the introduction, this characterization is a topological counterpart of the result saying that a totally ergodic shift invariant probability measure is Bernoulli if and only if it is the $\dbarm$-limit of the sequence of its canonical Markov approximations.

In Section \ref{proximal},  we also show how to apply this corollary even in the case when natural approximations of our shift are not comparable via inclusion, hence does not form a descending chain of shift spaces.

\section{$\dbar$-approachability vs. $\dbar$-stability} \label{sec:dbar-stab}

We recall the notion of $\dbarm$-stability that has been recently introduced by Tim Austin \cite{Austin}. 
Then we show that it follows from the $\dbar$-shadowing property. Later we discuss some further properties of $\dbarm$-stable shifts.
\begin{defn}
A shift space $X \subseteq \FS$ is \emph{$\dbarm$-stable} if for every $\eps>0$ there is an open neighborhood $\mathscr{U}$ of $\Ms(X)$ in the weak$^*$ topology on $\Ms(\FS)$ such that
if $\nu\in\mathscr{U}$, then there is $\mu\in\Ms(X)$ with $\dbarm(\mu,\nu)<\eps$.
\end{defn}

Note that we use slightly different notation: Austin writes $\dbar$ instead of $\dbarm$.
Equivalently, a shift space $X$ is $\dbarm$-stable if  any shift-invariant measure which lives close enough
to $X$ in the weak$^*$ topology is actually close in Ornstein's $\dbarm$ metric on $\Ms(\alf^\infty)$ to a
shift-invariant measure supported on $X$. This observation (noted already in \cite{Austin}) is formulated as the next lemma for future reference.
We state it in terms of the natural basis $(U_n)_{n=1}^\infty$ of  open neighborhoods of a shift space $X$ with respect to the Hausdorff topology induced by the usual (product) topology on $\FS$ (the topology of the Hausdorff metric $\rho^H$), where for $n\ge 1$ we have
\[U_n(X)=\bigcup \{[u]\mid u\in \lang_n(X)\}.\]


\begin{lem}\label{lem:equiv_cond_stability_shift}
  A shift space $X \subseteq \FS$ is $\dbarm$-stable if, and only if, for any $\eps > 0$ there are $\delta > 0$ and $N\in\N$ such that
  $\dbarm(\nu,\Ms(X))<\eps$ whenever $\nu\in\Ms(\FS)$,  $n\ge N$, and $\nu(U_n(X))>1-\delta$.
\end{lem}

\begin{prop}\label{prop:dbar-shadowing-stab}
If a shift space $X\subseteq\FS$ has the $\dbar$-shadowing property, then $X$ is $\dbarm$-stable.
\end{prop}
\begin{proof}
  Fix $1>\eps>0$. Use the definition of the $\dbar$-shadowing property to pick $n$ such that every sequence $(w^{(j)})_{j=1}^\infty$
in $\lang(X)$ with $|w^{(j)}|\ge n$ for every $j\ge 1$ is $\eps/2$-traced in the $\dbar$ pseudometric by some point in $X$. Let $\nu$ be a shift-invariant measure on $\FS$ such that $\nu(U_n(X))>1-\eps/2$. We would like to show that $\dbarm(\nu,\Ms(X))<\eps$.

Let $y=(y_i)_{i=0}^\infty\in\alf^\infty$
be a generic point for $\nu$ (such a point always exists in $\FS$). Hence, the frequency of visits of $y$ to in $U_n(X)$ satisfies
\[
\lim_{N\to\infty}\frac1N\{0\le k<N : \sigma^k(y)\in U_n(X) \}=\nu(U_n(X))> 1-\eps/2.
\]
Define $m_0\ge 0$ as the smallest $\ell\ge 0$ such that $y_{[\ell,\ell+n)}\in\lang_n(X)$. Inductively, given $m_k$ for some $k\ge 0$, we define  $m_{k+1}$ as the smallest integer $\ell\ge m_k+n$ such that $y_{[\ell,\ell+n)}\in\lang_n(X)$. Then the set
$$M=\N_0\setminus\bigcup_{k=0}^\infty [m_k,m_{k+1})$$
is contained in the set
$$\{\ell\in\N_0 \mid y_{[\ell,\ell+n)}\not\in\lang_n(X)\},$$
so its upper density is less than $\eps/2$. For $k\in\N_0$, we extend the word $y_{[m_k,m_{k}+n)}\in\lang_n(X)$ to the right to form some word $v^{(k)}\in \lang(X)$ of length $m_{k+1}-m_k$. Then the sequence
$$z=y_{[0,m_0)}v^{(0)}v^{(1)}\ldots$$
differs from $y$ only at positions (indices) belonging to the set  $M$. Hence $\dbar(y,z)\le\eps/2$. The same is true for $y'=\sigma^{m_0}(y)$ and $z'=\sigma^{m_0}(z)$. Furthermore, $y'$ is still generic for $\nu$. By the $\dbar$-shadowing property,
$$z'=v^{(0)}v^{(1)}\ldots$$
can be approximated by $z''\in X$ such that $\dbar(z',z'')<\eps/2$. Therefore,
$$\dbar(y',z'')<\eps/2+\eps/2=\eps.$$

It is a standard fact (see \cite[Thm. 15.23]{Glasner}) that every measure generated by $z''$ is $\eps$-close with respect to the $\dbarm$ distance to the measure generated by $y'$. Hence
$\dbarm(\nu,\Ms(X))\le \eps$.
\end{proof}



Shift spaces with the specification property are primary examples of $\dbar$-stable shift spaces provided by \cite{Austin}.
Recall that  a shift space $X\subset \FS$ has the \emph{specification property} if there exists $k\in\N$ such that for any $u,w\in\lang(X)$ there is $v$ with $|v|=k$ such that $uvw\in\lang(X)$. The specification property is a very useful property with many consequences for a shift space. For more extensive overview on the specification property and its relatives we refer the reader to \cite{KLO}. Here we note that the specification property and even either one of its two weaker, incommensurable variants known as \emph{almost specification property} or \emph{weak specification property}, see \cite{KLO},\footnote{The nomenclature is not fixed, we follow \cite{KLO}. Recall that a \emph{mistake function} is a nondecreasing function $g\colon\N\to\N$ with $g(n)/n\to 0$ as $n\to\infty$. We say that a shift space $X$ has the \emph{almost specification property} if there is a mistake function $g\colon\N\to\N$ such that for any $u,w\in\lang(X)$ there is a word $v\in\lang(X)$ satisfying $v=u'w'$, where $|u|=|u'|$, $|v|=|v'|$, and the following inequalities hold: \begin{align*}
|\{1\le j\le n: u_j\neq u'_j\}|&\le g(|u|),\\ 
|\{1\le j\le n: w_j\neq w'_j\}|&\le g(|w|). 
\end{align*}
We say that a shift space $X$ has the \emph{weak specification property} if there is a mistake function $g$ such that for any $u,w\in\lang(X)$ there is $v$ with $uvw\in\lang(X)$ satisfying $|v|=g(|w|)$. Weak and almost specification properties are independent of each other: neither one implies the other, see \cite{KOR}. Additionally, in contrast with the classical specification property, the weaker versions do not imply the uniqueness of the measure of maximal entropy, see \cite{KOR,Pavlov}.
}
imply $\dbar$-approachability and chain-mixing, hence $\dbar$-shadowing (\cite{KKK2}). Therefore we obtain the following corollary.

\begin{cor}\label{cor:specification-prop-imp-stability}
If $X$ is a shift space satisfying one of the following conditions
\begin{enumerate}
  \item $X$ has the specification property,
  \item $X$ has the weak specification property,
  \item $X$ has the almost specification property,
\end{enumerate}
then $X$ is $\dbarm$-stable.
\end{cor}

\begin{rem}
Using Proposition \ref{prop:dbar-shadowing-stab} we see that the proximal shift space constructed in Example \ref{ex:proximal} and the minimal shift space from Section \ref{sec:minimal} (see Theorem \ref{thm:exist_minimal_dshadowing}) are $\dbarm$-stable. This answers Austin's question in the affirmative way. We note that these examples do not have any of the specification properties mentioned in Corollary~\ref{cor:specification-prop-imp-stability}, because any of these specification properties implies that a shift having  one of them 
and positive topological entropy has many disjoint minimal proper subsets. Hence such a shift space is neither minimal nor proximal. 
\end{rem}
We list certain properties of $\dbarm$-stable shift spaces. But first we recall some definitions.
Let $X\subseteq \FS$ be a shift space.  The \emph{measure center} of $X$ is the smallest shift space $X^+\subseteq X$ such that  $\mu(X^+)=1$ for every $\mu\in\Ms(X)$. In other words, $X^+$ is the smallest subshift of $X$ containing supports of all invariant measures on $X$. The measure center is 
determined by the language of all words in $\lang(X)$ whose cylinders have positive measure for at one measure in $\Ms(X)$, that is 
\begin{equation*}
    \lang(X^+)=\left\{ w\in\lang(X) \mid \exists \mu \in \Ms(X) : \mu[w]>0 \right\}.
\end{equation*} The following observation follows directly from the definitions.
\begin{prop}\label{prop:dbar-stab-meas-center}
A shift space $X$ is $\dbarm$-stable if and only if its measure center $X^+$ is $\dbarm$-stable.
\end{prop}

It is possible that $X^+=X$ holds true for every $\dbarm$-stable shift space. In the next proposition, we note that $\dbarm$-stability of $X$ implies that the canonical Markov approximations of $X$ form a sequence of shift spaces satisfying \eqref{four-modes-1}. That is, $\dbarm$-stability implies that $X$ has a property we call \emph{$\dbarm$-approachability}.

\begin{prop}\label{prop:dstab-imp-dbarm-approach}
If $X\subseteq\FS$ is a $\dbarm$-stable shift space, then 
\[\Hdbarm(\Ms(X^M_n),\Ms(X))\to 0\quad\text{as} \quad n\to\infty.\]
\end{prop}
\begin{proof}
Fix $\eps>0$. Let $\delta>0$ and $n\ge 1$ be such that if $\nu\in\Ms(\FS)$ and $\nu(U_n(X))>1-\delta$, then $\dbarm(\Ms(X),\nu)<\eps$. Fix $m\ge n$ and $\mu\in\Ms(X^M_m)$. Since $X^M_m\subseteq U_n(X)$ we see that
$\dbarm(\mu,\Ms(X))<\eps$, so $\Hdbarm(\Ms(X^M_m),\Ms(X))<\eps$ for $m\ge n$. Hence, $\Hdbarm(\Ms(X^M_n),\Ms(X))\to 0$ as $n\to\infty$.
\end{proof}

The converse is not true, that is the condition $\Hdbarm(\Ms(X^M_n),\Ms(X))\to 0$ as $n\to\infty$ does not imply $\dbarm$-stability. Any  shift of finite type $X$ with non-transitive $X^+$ is a counterexample. A concrete example is the shift space $X$ over $\{0,1\}$ consisting of only two fixed points $0^\infty$ and $1^\infty$. It is a binary shift of finite type with $01$ and $10$ as the forbidden words. Such a shift space (trivially) satisfies $\Hdbarm(\Ms(X^M_n),\Ms(X))\to 0$ as $n\to\infty$, but is not $\dbarm$-stable because of Proposition \ref{prop:dbar-stab-implies-ent-dens} below. 
Nevertheless, adding an assumption of chain mixing to $\dbarm$-approachability, we obtain $\dbarm$-stability.

\begin{prop}
If $X\subseteq\FS$ is a chain mixing  shift space with $\Hdbarm(\Ms(X^M_n),\Ms(X))\to 0$ as $n\to\infty$, then $X$ is $\dbarm$-stable. 
\end{prop}
\begin{proof}
Fix $\eps>0$. Find $N>0$ such that $\Hdbarm(\Ms(X^M_N),\Ms(X))\le \eps/2$. Since $X$ is chain mixing, $X^M_N$ is a topologically mixing shift of finite type so it also has $\dbar$-shadowing property (see \cite{KKK2}). Hence $X^M_N$ is $\dbarm$-stable by Proposition \ref{prop:dbar-shadowing-stab}. Let $\delta>0$ and $m\ge 1$ be such that if $\nu\in\Ms(\FS)$ and $\nu(U_m(X^M_N))>1-\delta$, then $\dbarm(\Ms(X_N^M),\nu)<\eps/2$. Without loss of generality we have  $m\ge N$, so $U_m(X)\subseteq U_m(X^M_N)$. 
Let $\mu\in\Mse(\FS)$ be such that $\mu(U_m(X))>1-\delta$. Then $\mu(U_m(X^M_N))\ge \mu(U_m(X)) >1-\delta$. Hence $\dbarm(\mu,\Ms(X^M_N))<\eps/2$, so 
there exists $\mu'\in\Mse(X^M_N)$ such that $\dbarm(\mu,\mu')<\eps/2$. Since
$\Hdbarm(\Ms(X^M_N),\Ms(X))<\eps/2$ there exists $\xi\in\Mse(X)$ with $\dbarm(\xi,\mu)\le\dbarm(\xi,\mu')+\dbarm(\mu',\mu)<\eps$. 
\end{proof}

We do not know if $\dbarm$-stability implies the stronger condition called $\dbar$-approachability (the sequence of canonical Markov approximations of our shift space is a sequence of shift spaces satisfying condition \eqref{four-modes-4}), even if we assume that the shift space is chain-transitive or chain-mixing. The examples discussed in Proposition \ref{prop:diff-limits} suggest that it might not be the case.

Next, we note that $\dbarm$-stability implies weak$^*$ density of ergodic measures in $\Ms(X)$ and, as a consequence, entropy density (interestingly, we need to prove weak$^*$ density first to obtain transitivity of $X^+$ and obtain entropy density as a consequence of transitivity of $X^+$). Hence if $X$ is a $\dbarm$-stable shift space then the simplex of invariant measures $\Ms(X)$ is either a Poulsen simplex or a singleton. Note that the latter possibility can occur. As an example, take $X=\{0^\infty\}$.

\begin{prop}\label{prop:dbar-stab-implies-ent-dens}
If $X\subseteq\FS$ is a $\dbarm$-stable shift space, then $X^+$ is transitive and ergodic measures are entropy dense in $\Ms(X)$.
\end{prop}
\begin{proof} We first prove that $\Mse(X)$ is weak$^*$ dense  in $\Ms(X)$. 
To prove the density of ergodic measures it is enough to show that for every $\mu_1,\mu_2\in\Mse(X)$ the measure $\frac12(\mu_1+\mu_2)$ is a limit of a sequence of ergodic measures in $\Mse(X)$. Fix $\eps>0$. Let $\delta>0$ and $n\ge 1$ be such that if $\mu(U_n(X))>1-\delta$, then $\dbarm(\Ms(X),\mu)<\eps$. Since the ergodic measures of $\FS$ are dense in $\Ms(\FS)$, when the latter space is endowed with the weak$^*$ topology we can find $\nu\in\Ms(\FS)$ with $D(\frac12(\mu_1+\mu_2),\nu)$ as small as necessary. Here $D$ stands for any metric on $\Ms(\FS)$ compatible with the weak$^*$ topology. In particular, we may assume that $D(\frac12(\mu_1+\mu_2),\nu)$ is sufficiently small to guarantee $\nu(U_n(X))>1-\delta$. By $\dbarm$-stability, there is $\xi\in\Ms(X)$ such that $\dbarm(\nu,\xi)<\eps$. Since $\nu$ is ergodic, we can assure that $\xi$ is an ergodic measure. By the triangle inequality, $D(\frac12(\mu_1+\mu_2),\xi)\le D(\frac12(\mu_1+\mu_2),\nu)+D(\nu,\xi)$. Since $\nu$ and $\xi$ can be arbitrarily close in $\dbarm$, they can be also arbitrarily close in $D$. Hence $\xi$ can be arbitrarily close to $\frac12(\mu_1+\mu_2)$ in $D$ and the ergodic measures must be weak$^*$ dense. Now, transitivity of $X^+$ follows easily from weak$^*$ density of ergodic measures (see Proposition 6.4 in \cite{GK} for details). 
By Proposition \ref{prop:dbar-stab-meas-center} the measure center $X^+$ is a $\dbarm$-stable shift space. Now, Proposition \ref{prop:dstab-imp-dbarm-approach} implies $\Hdbarm(\Ms((X^+)^M_n),\Ms(X^+))\to 0$ as $n\to\infty$. Transitivity of $X^+$ imply that its Markov approximations are entropy dense shift spaces. Hence ergodic measures are entropy dense in
$\Ms(X^+)$ by \cite{KKK2}, but clearly $\Ms(X)=\Ms(X^+)$.
\end{proof}



\begin{prop}
If $X$ is a strictly ergodic $\dbarm$-stable shift space, then the unique invariant measure on $X$ is isomorphic to an odometer. 
\end{prop}
\begin{proof}
Assume that $X$ is a strictly ergodic infinite shift space. Let $\nu$ be its unique ergodic invariant measure, that is, $\Ms(X)=\{\nu\}$. Hence, for every $n\ge 1$ the Markov approximation  $X_n^M$ of $X$ is an uncountable shift of finite type. In particular, for every $n\ge 1$ the simplex  $\Ms(X_n^M)$ contains infinitely many periodic ergodic invariant measures (measures concentrated on periodic orbits). 
Now assume that $X$ is $\dbarm$-stable, so $\Hdbarm(\Ms(X^M_n),\Ms(X))\to 0$ as $n\to\infty$. It follows that 
$\lim_{n\to\infty}\dbarm(\mu^{\text{per}}_n,\nu)=0$ for any choice of periodic ergodic measures  $\mu_n^{\text{per}}\in \Ms(X^M_n)$. In particular, one may take measures on periodic points whose primary periods tend to infinity. A measure that is a $\dbarm$-limit of such a sequence of periodic measures must be isomorphic to a Haar measure on some odometer (this result is implicit in \cite{RS} and follows directly from \cite[Thm.\ 1.7]{BKPLR}, it does not need invertibility assumption).  
\end{proof}
\begin{rem}[Some open questions]
The results in the present section do not provide a complete picture of connections between the notions of $\dbarm$-stability and $\dbar$-approachability. For example, we were unable to answer the following questions:  
Can a non-trivial periodic orbit be $\dbarm$-stable shift space?
Can a strictly ergodic infinite shift space be $\dbarm$-stable? Is every $\dbarm$-stable system topologically mixing on its measure center? Can a shift space $X$ such that $X^+\neq X$ be $\dbarm$-stable?  
\end{rem}

\section{Comparing $\Hdbarm$ with $\Hdbar$} \label{sec:Hdabr-vs-Hdbarm}

If $(Z,\rho)$ is a bounded complete metric space, then so is $(\CL(Z),\rho^H)$ (see \cite[\S2.15]{IN}). Hence the Hausdorff metric $\Hdbarm$ induced on $\CL(\Ms(\FS))$ by $\dbarm$ is complete and Cauchy condition provides a criterion for convergence of a sequence $(\Ms(X_k))_{k=1}^\infty$, where $X_k\subseteq\FS$ is a shift space for every $k\ge 1$.
But even if we know that $(\Ms(X_k))_{k=1}^\infty$ converges in $\dbarm$ to some $\M\in\CL(\Ms(\FS),\dbarm)$, it is not clear if there exists a shift space $X\subseteq\FS$ such that $\M=\Ms(X)$.
We provide an example showing that this need not to be the case at the end of this section.
But first we demonstrate that shift spaces $X$ and $Y$ with the $\dbar$-shadowing property are  $\Hdbar$ close if and only if they are $\Hdbarm$ close. 
On the other hand, without the $\dbar$-shadowing property the inequality in \eqref{hdbar-ineq} can be strict. We use a variant of Oxtoby's construction of nonuniquely ergodic minimal Toeplitz subshift to show that for every $\delta>0$ there are shift spaces $X$ and $Y$ such that
$\Hdbarm(\Ms(X), \Ms(Y))<\delta$ but $\Hdbar(X,Y)>1-\delta$. Finally, we show that if the measure center of a shift space $X$ has the $\dbar$-shadowing property, then $X$ also has it.


\begin{thm}\label{thm:dbarshad-dbarm-dbar}
If $X$ and $Y$ are shift spaces over $\alf$ with the $\dbar$-shadowing property such that \[\Hdbarm\left(\Ms(X),\Ms(Y)\right)<\eps^2\] for some $\eps>0$, then $\dbar^H(X,Y)<7\eps$.
\end{thm}
\begin{proof}
Fix 
$x\in X$. Use $\dbar$-shadowing of $Y$ to find $s\in \N$ such that for every sequence $\{ w^{(j)} \}_{j=1}^\infty$ of words in $\lang\left(Y\right)$ with $|w^{(j)}|\ge s$ for every $j\geq1$, there exists $y\in Y$ such that
\begin{equation}
    \dbar ( w^{(1)}w^{(2)}w^{(3)}\ldots , y) < \eps.
\end{equation}
Pick $m$ such that $s<m\eps$. By \cite[Theorem 3.4]{DownarowiczWiecek} we find $l\geq m$ such that $x$ can be decomposed into infinite concatenation of blocks, that is, we can write
\begin{equation}\label{decompositionofx}
    x= A^{(1)}B^{(1)}A^{(2)}B^{(2)}\ldots.
\end{equation} and the blocks $A^{(1)}, A^{(2)}, \ldots $ and $B^{(1)},B^{(2)},\ldots$ satisfy
\begin{itemize}
\item for every $i\geq 1$ we have 
satisfies $m\leq |B^{(i)}| \leq l$;
\item for every $i\geq 1$  there exists an ergodic measure $\mu^{(i)}\in\Mse(X)$ such that
\begin{equation}\label{eq:d-star}
d^*(B^{(i)},\mu^{(i)})=\sum_{k=1}^\infty2^{-k}\sum_{w\in\alf^k}\left|\freq(w,B^{(i)})-\mu^{(i)}([w])\right|<\dfrac{\eps}{2^{s}},
\end{equation}
where
\begin{equation}\label{eq:freq-def}
\freq(w,B^{(i)})=\begin{cases}\frac{|\{1\le j \le |B^{(i)}|-l+1 \ : \ B^{(i)}_{[j,j+l)}=w\}|}{|B^{(i)}|}
,&\text{if }|w|=l\le |B^{(i)}|,\\
0,&\text{otherwise;}
\end{cases}
\end{equation}
\item the set of coordinates of $x$ which belong to the block $A^{(i)}$ in \eqref{decompositionofx} for some $i\ge 1$ has upper Banach density smaller than $\eps$.
In particular we have
\begin{equation}\label{ineq:dbar-|A|}
  \limsup_{n\to\infty}\frac{|A^{(1)}|+\ldots|A^{(n)}|}{|A^{(1)}B^{(1)}|+\ldots+|A^{(n)}B^{(n)}|}<\eps.
\end{equation}
\end{itemize}
Now, we use the assumption $\Hdbarm\left(\Ms(X),\Ms(Y)\right)<\eps^2$ and for every $i\geq 1$  
we find an ergodic measure $\nu^{(i)}\in \Ms(Y)$ such that
\begin{equation}
    \dbarm(\mu^{(i)},\nu^{(i)})<\eps^2.
\end{equation}

Following Shields \cite{Shields}, for every $\mu,\nu\in\Ms(\FS)$ and $n\ge 1$ we define $J_n=J_n(\mu,\nu)$ to be the set of measures $\lambda_n$ on $\alf^n\times\alf^n$ endowed with the powerset $\sigma$-algebra such that
for every $u,w\in\alf^n$ we have $\mu[u]=\lambda_n(\{u\}\times \alf^n)$ and $\nu[w]=\lambda_n(\alf^n\times\{w\})$. For $\alpha>0$ we let
$\Delta_n(\alpha)=\{(u,w)\in\alf^n\times\alf^n:\dHam(u,w)\le\alpha\}$. Finally, for $\mu,\nu\in\Ms(\FS)$ we write
\[
d^*_n(\mu,\nu)=\max_{\lambda_n\in J_n(\mu,\nu)}\min\{\alpha>0:\lambda_n(\Delta_n(\alpha)\ge 1-\alpha\}.
\]
By \cite[Section I.9]{Shields}, in particular \cite[Lemma I.9.12]{Shields}, we see that $\dbarm(\mu^{(i)},\nu^{(i)})<\eps^2$ implies that for every $n\ge 1$ we have $d^*_n(\mu^{(i)},\nu^{(i)})<\eps$. Hence, for every $n\ge 1$ and $i\ge 1$ there exists $\lambda^{(i)}_n\in J_n(\mu^{(i)},\nu^{(i)}))$ such that $\lambda^{(i)}_n(\Delta_n(\eps))>1-\eps$. 
Consider the set 
\[G_n^{(i)}=\{u\in\lang_n(X):\text{there exists }w\in\lang_n(Y)\text{ with }\lambda_n^{(i)}(\{(u,w)\}\cap\Delta_n(\eps))>0\}.\]
It follows that for every $u\in G_n^{(i)}$ we can pick $w_n^{(i)}(u)$ such that $\dHam(u,w^{(i)}_n(u))<\eps$ and
$\lambda_n^{(i)}(\{(u,w^{(i)}_n(u))\})>0$. In particular, $\nu^{(i)}[w^{(i)}_n(u)]>0$ and hence $w^{(i)}_n(u)\in\lang_n(Y)$.
In addition, we clearly have
\[
\lambda^{(i)}_n(\Delta_n(\eps))=\lambda^{(i)}_n\left(\Delta_n(\eps)\cap (G_n^{(i)}\times\alf^n)\right)>1-\eps.
\]
By an abuse of notation, for $i\ge 1$ and $n\ge 1$, by $\bigcup G^{(i)}_n$ we will understand $\bigcup\{[u]:u\in G_n^{(i)}\}$.

It follows that for every $n\ge 1$ and $i\ge 1$ we have
\begin{equation}\label{eq:mu-gi}
\mu^{(i)}\left(\bigcup G_n^{(i)} \right)=\lambda^{(i)}_n(G_n^{(i)}\times\alf^n)\ge \lambda^{(i)}_n(\Delta_n(\eps)\cap G_n^{(i)}\times\alf^n)>1-\eps.
\end{equation}
For each $i\geq 1$ we take $n=s$ and consider $G_s^{(i)}\subseteq \lang_s(X)$. Note that \eqref{eq:mu-gi} implies that $\mu^{(i)}(\bigcup G_s^{(i)})>1-\eps$.
In analogy with \eqref{eq:freq-def}, we define $\freq(G_s^{(i)},B^{(i)})$ to be number of coordinates in $B^{(i)}$ where some word from $G^{(i)}_s$ appears in $B^{(i)}$ divided by the length of $B^{(i)}$, that is,
\[
\freq(G_s^{(i)},B^{(i)})=\frac{\left|\{1\le j\le |B^{(i)}|-s+1: B^{(i)}_{[j,j+s)}\in G_s^{(i)}\}\right|}{|B^{(i)}|}.
\]
We easily see that
\[
\freq(G_s^{(i)},B^{(i)})=\sum_{w\in G_s^{(i)}}\freq(w,B^{(i)}).
\]
Let $P$ be the set of coordinates in $B^{(i)}$  covered by occurrences of words from $G^{(i)}_s$  in $B^{(i)}$, that is
\[
P=\{1\le p\le |B^{(i)}|: \exists 1\le j\le |B^{(i)}|-s+1 \text{ with }B^{(i)}_{[j,j+s)}\in G_s^{(i)}\text{ and }j\le p<j+s\}.
\]
We clearly have
\begin{equation}\label{ineq:freq-gi}
\freq(G_s^{(i)},B^{(i)})\le \frac{\left|P\right|}{|B^{(i)}|}.
\end{equation}
Furthermore
\begin{equation}\label{ineq:z}
\left|\mu^{(i)}\left(\bigcup G_s^{(i)}\right)-\freq(G_s^{(i)},B^{(i)})\right|\le \sum_{w\in G^{(i)}_s}\left|\freq(w,B^{(i)})-\mu^{(i)}([w])\right|.
\end{equation}
Using $d^*\left(B^{(i)},\mu^{(i)}\right)<\dfrac{\eps}{2^{s}}$ (cf. \eqref{eq:d-star}) we obtain
\begin{align}\label{ineq:zz}
\sum_{w\in G^{(i)}_s}\left|\freq(w,B^{(i)})-\mu^{(i)}([w])\right| 
& \le \sum_{w\in\alf^s}\left|\freq(w,B^{(i)})-\mu^{(i)}([w])\right|
\\ & \le 2^sd^*(B^{(i)},\mu^{(i)})<\eps.
\end{align}
Combining  \eqref{ineq:z} and \eqref{ineq:zz} we get
\begin{equation}\label{ineq:mu-gi-freq}
\left|\mu\left(\bigcup G_s^{(i)}\right)-\freq(G_s^{(i)},B^{(i)})\right|\le \eps.
\end{equation}
Combining \eqref{eq:mu-gi} and \eqref{ineq:mu-gi-freq} we see that
\[
1-2\eps\le  \frac{\left|P\right|}{|B^{(i)}|}.
\]
It follows that there exists a decomposition of $B^{(i)}$ 
such that
\begin{equation}\label{decompositionofB_i}
    B^{(i)}=v^{(i,1)}u^{(i,1)}v^{(i,2)}u^{(i,2)}\ldots v^{(i,\kappa(i))} u^{(i,\kappa(i))}v^{(i,\kappa(i)+1)},
\end{equation}
where $\kappa(i)$ is some (large) positive integer, $u^{(i,j)}\in G^{(i)}_s$ and $v^{(i,j)}\in\lang(X)\setminus G^{(i)}_s$. 
Furthermore,
\begin{equation}\label{ineq:|v_ij|}
\frac{|v^{(i,1)}|+|v^{(i,2)}|+\ldots+|v^{(i,\kappa(i))}|+|v^{(i,\kappa(i)+1)}|}{|B^{(i)}|}\le\frac{|B^{(i)}|-|P|}{|B^{(i)}|}\le 2\eps.
\end{equation}
(Actually, \eqref{ineq:|v_ij|} tells us that $v^{(i,j)}$ is an empty word for many $j$'s.)





We claim that if for each $i\geq 1$ we find blocks $\hat{w}^{(i,j)}\in\lang(Y)$ ($j=1,\ldots,\kappa(i)$) with $|\hat{w}^{(i,j)}|\ge s$ for each $j$ and these blocks will satisfy
\begin{equation}\label{eq:|AB|}
|A^{(i)}B^{(i)}|=|\hat{w}^{(i,1)}\ldots \hat{w}^{(i,\kappa(i))}|
\end{equation}
and
\begin{equation}\label{eq:Ham-hat_w}
\dHam\left(A^{(i)}B^{(i)},\hat{w}^{(i,1)}\ldots \hat{w}^{(i,\kappa(i))}\right)\le \frac{|A^{(i)}|+2\eps|B^{(i)}|+|v^{(i,1)}|+\ldots+|v^{(i,\kappa(i)+1)}|+s}{|A^{(i)}B^{(i)}|},
\end{equation}
then the proof will be complete. 
Indeed, assume that for each $i\ge 1$ we have found blocks $\hat{w}^{(i,1)},\ldots,\hat{w}^{(i,\kappa(i))}$ satisfying \eqref{eq:|AB|} and \eqref{eq:Ham-hat_w}, each of them of length at least $s$. We set $\hat{y}$ to be infinite concatenation of $\hat{w}^{(i,1)},\ldots,\hat{w}^{(i,\kappa(i))}$, where $i=1,2,\ldots$, that is
\[
\hat{y}=\hat{w}^{(1,1)}\ldots\hat{w}^{(1,\kappa(1))}\hat{w}^{(2,1)}\ldots\hat{w}^{(2,\kappa(2))}\ldots\ldots\ldots\hat{w}^{(i,1)}\ldots\hat{w}^{(i,\kappa(i))}\ldots
\]
By \eqref{decompositionofx}, \eqref{ineq:dbar-|A|}, \eqref{ineq:|v_ij|}, \eqref{eq:|AB|}, and \eqref{eq:Ham-hat_w} such $\hat{y}$ satisfies
\begin{equation}\label{ineq:hat-y-to-x}
\dbar(x, \hat{y})\le 6\eps.
\end{equation}
Note that for every $i\ge 1$ and for every $1\le j\le\kappa(i)$ we have $\hat{w}^{(i,j)}\in\lang(Y)$ and  $|\hat{w}^{(i,j)}|\ge s$, so the $\dbar$-shadowing property guarantees there is $y\in Y$ with
\begin{equation}\label{ineq:y-y-hat}
  \dbar(y,\hat{y})<\eps.
\end{equation}
By \eqref{ineq:hat-y-to-x} and \eqref{ineq:y-y-hat} we have $\dbar(x,y)<7\eps$ as needed. 

It remains to find appropriate $\hat{w}^{(i,1)},\ldots,\hat{w}^{(i,\kappa(i))}$ for each $i\ge 1$. To this end we fix $i\ge 1$ and for each $2\le j\le\kappa(i)$ we take $u^{(i,j)}$ in equation \eqref{decompositionofB_i} to find $w^{(i,j)}=w^{(i,j)}(u^{(i,j)})\in \lang_s(Y)$ with $\dHam\left(u^{(i,j)},w^{(i,j)}\right)\leq \eps$. Now, for $j=1$ we 
set $t(i)=|A^{(i)}|+|v^{(i,1)}|+|w^{(i,1)}|+|v^{(i,2)}|\ge s$ and we pick any $\hat{w}^{(i,1)}\in \lang_{t(i)}(Y)$. 
For $2\le j\le \kappa(i)$ we simply extend each $w^{(i,j)}$ to a word  $\hat{w}^{(i,j)}=w^{(i,j)}\hat{v}^{(i,j+1)}\in \lang_{|w^{(i,j)}|+|v^{(i,j+1)}|}(Y)$, where $|\hat{v}^{(i,j+1)}|=|v^{(i,j+1)}|$. 
We clearly have $|\hat{w}^{(i,j)}|\ge |w^{(i,j)}|=s$ and  
\[
\dHam(\hat{w}^{(i,j)},u^{(i,j)}v^{(i,j+1)})\le \frac{|v^{(i,j+1)}|+\eps|w^{(i,j)}|}{|\hat{w}^{(i,j)}|}.
\]
It is now straightforward to see that the blocks $\hat{w}^{(i,1)},\ldots,\hat{w}^{(i,\kappa(i))}$ satisfy \eqref{eq:|AB|} and \eqref{eq:Ham-hat_w}.
To finish the proof reverse the roles of $X$ and $Y$. \end{proof}

\subsection{Examples}
In this subsection, we explore what happens if we abandon the assumption of $\dbar$-shadowing.
First, we prove  Proposition \ref{prop:no_simplex} showing that the $\Hdbarm$ limit of a sequence of simplices of invariant measures need not be a simplex of all invariant measure of some subshift. In particular, the shift spaces $X_k$ consisting of a single periodic orbit that have been constructed in the course of the proof of Proposition \ref{prop:no_simplex} do not converge in $\Hdbar$-distance to a shift space, as their convergence to a shift $X$ would imply the limit of the corresponding simplices would be $\Ms(X)$, see \cite{KKK2}.

\begin{prop}\label{prop:no_simplex}
For every alphabet $\alf$ there exists a sequence of transitive finite shifts 
$(X_k)_{k=1}^\infty$ such that
for some ergodic fully supported measure $\mu\in\Mse(\FS)$ we have
\[\Hdbarm(\Ms(X_k),\{\mu\})\to 0\quad\text{as } k\to \infty. \]
In particular, there does not exist a shift space $X$ such that $\{\mu\}=\Ms(X)$.
\end{prop}
\begin{proof}
We order the nonempty words over $\alf$ into a sequence $(W_k)_{k=0}^\infty$. Let $(\delta_k)_{k=1}^\infty$ be a sequence of positive reals such that \begin{equation}\label{sum-conv}
  \sum_{k=1}^\infty \delta_k<1/2.
\end{equation} We inductively define words $(V_k)_{k=0}^\infty$ by
\[V_0=W_0,\qquad V_{k+1}=V_k^{a_{k+1}}W_{k+1}1^{b_{k+1}}\quad\text{for }k\ge 0\]
in such a way that 
$b_{k+1}\ge 0$ is the smallest number such that $|W_{k+1}|1^{b_{k+1}}$ is a multiple of $|V_k|$ and
$a_{k+1}\ge 1$ is chosen so that the following inequality holds true
\begin{equation}\label{ineq:vk}\frac{|W_{k+1}|+b_{k+1}}{|V_{k+1}|}<\delta_{k+1}.\end{equation}
This implies that
$c_k=|V_{k+1}|/|V_k|$ is a positive integer. For $k\ge 1$, let $x^{(k)}=V_k^{\infty}\in\FS$ be a periodic point, $X_k$ be its orbit, and $\mu_k$ be the unique ergodic measure of the shift space $X_k$. 

Note that $\mu_k[W_k]\ge 1/|V_{k}|>0$ and for every $n>k$ we have
\begin{equation}\label{ineq:prod}
\mu_n[W_k]\ge \frac{1}{|V_{k}|}\left((1-\delta_{k+1})(1-\delta_{k+2})\cdot\ldots\cdot(1-\delta_n))\right)>0,
\end{equation}
because the infinite product $(1-\delta_{k+1})(1-\delta_{k+2})\cdot\ldots\cdot(1-\delta_n)\ldots$ converges to a non-zero limit by \eqref{sum-conv}.
  We constructed the words $V_k$ in such a way that $|V_{k+1}|$ is a multiple of $|V_k|$ and that $V_k^{a_{k+1}}$ is a prefix of $V_{k+1}$ for every $k\ge 0$, hence using \eqref{ineq:vk} we see that for every $k\ge q$ we have  \[\dbar(x^{(k+1)},x^{(k)})=
  \dHam(V_k^{a_{k+1}}W_{k+1}1^{b_{k+1}},V_k^{c_k})\le \delta_{k+1}.
  \] 

It follows that $x^{(k)}$ is a Cauchy sequence in $\dbar$ pseudometric. Since $\dbar$ is a complete pseudometric the sequence $x^{(k)}$ converges, so there is $x\in\FS$ such that
\[
\lim_{k\to\infty} \dbar(x,x^{(k)})=0.\]
This point $x$ must be then generic for an ergodic shift invariant measure $\mu$ such that $\mu$ is $\dbarm$-limit of the measures $\mu_k$. Since $\dbarm$-convergence implies weak$^*$ convergence, the portmanteau theorem and \eqref{ineq:prod} imply that
\[
\mu[W_k]=\lim_{n\to\infty} \mu_n[W_k]\ge \frac{1}{|V_{k}|}\prod_{n=1}^{\infty}(1-\delta_{n})>0.
\]
Since $\mu[W_k]>0$ for every $k\ge 0$, the only subshift $X$ such that $\mu\in\Ms(X)$ is the full shift. On the other hand, $\Ms(X_k)=\{\mu_k\}$ and $\Hdbarm(\Ms(X_k),\{\mu\})=\dbar(\mu_k,\mu)\to 0$ as $k\to\infty$. We see that $\Ms(X_k)=\{\mu_k\}$ converge to $\{\mu\}$ in  $\Hdbarm$ but there is no subshift $X$ of $\FS$ such that $\Ms(X)=\{\mu\}$.
\end{proof}

Our next goal is to show an example of a sequence $(X_k)_{k=1}^\infty$ of shift spaces such that the $\Hdbarm$-limit of simplices $\Ms(X_k)$ exists and is a simplex of invariant measures of some shift space $X$, but the shift spaces $X_k$ do not converge to $X$ with respect to $\Hdbar$ pseudometric.

To find our examples we will adapt the construction of one-sided Oxtoby's sequences. The original  Oxtoby's sequence generates a minimal non-uniquely ergodic Toeplitz subshift, see \cite{Downarowicz, Oxtoby,Wil}. As parameters of this construction we need a sequence of positive integers $(p_k)_{k=0}^\infty$.

\begin{defn}\label{def:Oxtoby}
Let $\mathbf{p}=(p_k)_{k=0}^\infty$ be a sequence of positive integers such that $p_0=1$, and for each $k\ge 0$ we have that $p_k$  divides $p_{k+1}$ and $p_{k+1}/p_k\ge 3$. Let $M_0=\emptyset$, and for $k\ge 1$ define $M_k=([-p_k,p_k)+p_{k+1}\N)\cap\N$. Note that for every $i\in \N$ there exists a unique $k=k(i)\ge 1$ such that $i\in M_{k}\setminus \bigcup^{k-1}_{\ell=0} M_\ell$. We define the \emph{Oxtoby sequence with the scale $\mathbf{p}$} to be a binary sequence $x(\mathbf{p})\in\{0,1\}^\infty$ such that 
$x(\mathbf{p})_i=k(i) \bmod 2$.
\end{defn}

By Lemma 3.2 in \cite{Wil}, if 
$x(\mathbf{p})\in\{0,1\}^\infty$ is an Oxtoby sequence with scale $\mathbf{p}$ satisfying
\[
\sum^\infty_{k=0}\frac{p_{k}}{p_{k+1}}<\infty,
\]
then the orbit closure of $x(\mathbf{p})$ in $\{0,1\}^\infty$ is a minimal shift space $X(\mathbf{p})$ with exactly two ergodic invariant measures.

\begin{prop}\label{prop:diff-limits}
If $\mathbf{p}=(p_k)_{k\in\N}$ is a sequence of positive integers satisfying
\begin{equation}\label{cond:scale}\sum^\infty_{k=1}\frac{2p_{k}}{p_{k+1}}<\delta,\end{equation}
for some $0<\delta<1/2$, then the minimal shift $X(\mathbf{p})$ obtained as the orbit closure of the Oxtoby sequence with scale $\mathbf{p}$ satisfies
$$\Hdbar(X(\mathbf{p}),\{0^{\infty},1^{\infty}\})>1-\delta,\qquad\text{and}\qquad \Hdbarm(\Ms(X(\mathbf{p})),\Ms(\{0^{\infty},1^{\infty}\}))<\delta.$$
\end{prop}
\begin{proof}
Fix $0<\delta<1$ and a sequence of positive integers $\mathbf{p}$ as above. For simplicity we write $x$ for the Oxtoby sequence $x(\mathbf{p})$ defined taking $\mathbf{p}$ as its scale and $X$ for the associated minimal subshift $X(\mathbf{p})$ (see Definition \ref{def:Oxtoby}). By Lemma 3.2 in \cite{Wil},  $\Mse(X)=\{\mu',\nu'\}$. Let $(M_k)_{k=1}^\infty$ be a sequence of sets as in Definition \ref{def:Oxtoby}.
Fix $k\ge 1$ and consider the prefix $x_{[0,p_{k+1})}$. Since $p_{k+1}$ is a multiple of $p_\ell$ for every $\ell\le k$, using the structure of the sets $M_k$ for $k\ge \ell$ we get that
\[\left|M_\ell\cap [0,p_{k+1})\right|=\frac{p_{k+1}}{p_{\ell+1}}\cdot 2 p_{\ell}.\]
  Hence, 
\[\frac{\left|(\bigcup^{k}_{\ell=0} M_\ell)\cap [0,p_{k+1})\right|}{p_{k+1}}\le \sum^{k}_{\ell=0}\frac{2p_{\ell}}{p_{\ell+1}}\le \delta.\]
But for every $i\in [0,p_{k+1})\setminus \bigcup^{k}_{\ell=0} M_\ell$ we have $x_i={k+1} \bmod 2$. In other words, the Oxtoby sequence is constant for all indices $i$ in $[0,p_{k+1})\setminus \bigcup^{k}_{\ell=0} M_\ell$ with the constant depending only on the parity of $k$.
Since
\[\frac{\left|([0,p_{k+1})\cap\N \setminus \bigcup^{k}_{\ell=0} M_\ell) \right|}{p_{k+1}}\ge
1-\delta,\]
we see that for both $\alpha=0$ and $\alpha=1$ we have
\[
\limsup_{k\to\infty}\frac{\left|\{0\le i < p_{k+1}:x_i=\alpha\}\right|}{p_{k+1}}\ge
1-\delta.
\]
Hence, for some invariant measure $\mu,\nu\in\Ms(X)$ we have $\mu([0])>1-\delta$ and $\nu([1])>1-\delta$.
By ergodic decomposition, these measures are convex combinations of the ergodic measures $\mu'$ and $\nu'$. 
This implies that for one ergodic measure, say $\mu'$, we have $\mu'([0])>1-\delta$, while for the other one, $\nu'([1])>1-\delta$. 

  A generic point for $\mu'$ has density of ones at most $\delta$, so it is at most $\delta$ far away from $0^\infty$. Hence $\dbarm(\mu',\dirac_{0^\infty})<\delta$. Similarly, $\dbarm(\nu',\dirac_{1^\infty})<\delta$. Therefore the $\Hdbarm$-distance between sets of ergodic measures on $X$ and $\{0^{\infty},1^{\infty}\}$ is bounded by $\delta$. By \cite[Lemma 14]{KKK2} (see \eqref{hdbar-ineq} in the introduction) we have $\Hdbarm(\Ms(X(\mathbf{p})),\Ms(\{0^{\infty},1^{\infty}\}))=\Hdbarm(\Mse(X(\mathbf{p})),\Mse(\{0^{\infty},1^{\infty}\}))<\delta$.
  On the other hand, by the above calculations, we see that the Oxtoby sequence $x$ satisfies
  $\dbar(x,1^\infty)>1-\delta$ and $\dbar(x,0^{\infty})>1-\delta$, which means that  $\Hdbar(X(\mathbf{p}),\{0^{\infty},1^{\infty}\})>1-\delta$.
\end{proof}

\begin{cor}
There exists a sequence $(X_k)_{k=1}^\infty$ of minimal shift spaces such that for some shift space $X$ we have
$\Hdbarm(\Ms(X_n),\Ms(X))\to 0$ while $\Hdbar(X_n,X)\to 1$ as $n\to\infty$.
\end{cor}
\begin{proof}
  One can take $X=\{0^{\infty},1^{\infty}\}$ and sequence of minimal shift spaces $X_n$ generated by Oxtoby sequences $x^{(n)}$ constructed in the previous proposition for a sequence of $\delta$'s going to zero.
\end{proof}

\subsection{On $\dbar$-shadowing on the measure center of $X$}
Let us recall, that the measure center $X^+$ of a shift space $X$ is the smallest subshift of $X$ containing supports of all invariant measures on $X$, that is, $X^+$ is the smallest closed set such that $\mu(X^+)=1$ for every $\mu\in\Ms(X)$.


Fix a word  $u$ over $\alf$ with $|u|\geq k$. Given a word $w$ over $\alf$ we define $\gamma_w(u)$ to be the number of occurrences of $w$ in $u$, that is,
\begin{equation*}
    \gamma_w(u) = \big| \left\{ 1\leq j \leq |u|-k+1 \mid u_ju_{j+1}\ldots u_{j+k-1}=w \right\} \big|.
\end{equation*}
Furthermore, for $n\in\N$ with $n\ge |w|$ we set $\Gamma_w(n)$ to be largest number of occurrences of $w$ among all words $u$ of length $n$, that is
\begin{equation*}
    \Gamma_w(n)= \max \{ \gamma_w(u) \mid u\in\lang_n(X) \}.
\end{equation*}
It is a straightforward consequence of the definition that $\Gamma_w(uv)\le\Gamma_w(u)+\Gamma_w(v)+|w|-1$ for every $u,v,w\in\alf^*$. Thus
\begin{equation}
    \Gamma_w(n+m)\leq \Gamma_w(n)+\Gamma_w(m)+|w|-1
\end{equation}
for all $n,m$.

%
We define the maximum limiting frequency of $w$ in $X$ as
\begin{equation}
    \Lambda_X(w)=\lim_{n \rightarrow\infty} \dfrac{1}{n}\Gamma_w(n).
\end{equation}
The existence of the limit follows from
the subadditivity of the function $\Gamma'_w(n)=\Gamma_w(n)+|w|-1$ and the fact that the difference between ratios $\Gamma_w(n)/n$ and $\Gamma'_w(n)/n$ goes to zero.

It is known that
\begin{equation}
    \Lambda_X(w)=\max_{\mu\in\Ms(X)}\mu[w]=\max_{\nu\in\Mse(X)}\nu[w],
\end{equation} see \cite[Chap. 3]{Furstenberg}.
%
This means that $w\in \lang(X)\setminus\lang(X^+)$ if and only if $\Lambda_X(w)=0$, that is, for every $\eps > 0$ there exists $N\in\N$ such that for all $n\geq N$ for all $u\in\lang_n(X)$ we have $\gamma_w(u)\leq n\eps$.

\begin{thm}\label{thmonmeasurecenter}
If $X^+$ has the $\dbar$-shadowing property then so does $X$.
\end{thm}
\begin{proof}
Fix $\eps>0$. For $\eps/3$ we use the $\dbar$-shadowing property of $X^+$ to find $N_1$ such that for any sequence of words $\{ a^{(j)} \}_{j=1}^\infty$ in $\lang(X^+)$ with $|a^{(j)}|\geq N_1$ there exists $x\in X^+$ such that $\dbar\left(x,a^{(1)}a^{(2)}\ldots\right)<\eps/3$.

Now, fix $m\geq N_1$. For $w\notin \lang_m(X^+)$ let $N_w>0$ be such that for all $n\geq N_w$ for all $u\in\lang_n(X)$ we have
\begin{equation}\label{upperboundoncounting}
    \gamma_w(u)\leq \dfrac{n\eps}{|A|^m3m}.
\end{equation}
Set $N_0 = \max_{w\in\lang_m(X)}{N_w}$. We take $N\in\N$ such that ${m}/{N}<\eps/6$ and $N\geq \max\{N_0,N_1\}$.

Let us take any $j\geq 1$ and any $w^{(j)}\in\lang(X)$ such that $|w^{(j)}|\geq N$. Each $w^{(j)}$ can be written as a concatenation of finite blocks as follows;
\begin{equation*}
    w^{(j)}=u_1^{(j)}u_2^{(j)}\ldots u_{k(j)-1}^{(j)}u_{k(j)}^{(j)},
\end{equation*}
where $|u_i^{(j)}|=m$ for $1\leq i< k(j)$ and $m\leq |u_{k(j)}^{(j)}|<2m$. Using \eqref{upperboundoncounting} we see that for $1\leq i \leq k(j)-1$, the number of $u_i^{(j)}$'s which are not in $\lang_m(X^+)$ is bounded from above by $\dfrac{\eps|w^{(j)}|}{3m}$.

For each $j\geq1$, we create $\Bar{w}^{(j)}$ by replacing each $u_i^{(j)}\notin \lang(X^+)$ by some word $\bar{v}\in\lang_m(X^+)$ for $1 \leq i < k(j)$ and replacing $u_{k(j)}^{(j)}$ by some word $\bar{v}^{(j)}\in \lang(X^+)$ with $|u_{k(j)}|=|\bar{v}^{(j)}|$ if $u_{k(j)}^{(j)}\notin \lang(X^+)$. Therefore, we have
\begin{equation}\label{distanceof-barw-and-w}
    \dbar\left(\Bar{w}^{(1)}\Bar{w}^{(2)}\ldots,{w}^{(1)}{w}^{(2)}\ldots\right)<\dfrac{m\eps}{3m}+\dfrac{2\eps}{6}=\dfrac{2\eps}{3}.
\end{equation}
Notice that for each $j\geq 1$ the word $\Bar{w}^{(j)}$ is a concatenation of words from $\lang(X^+)$ whose lengths are greater or equal to $m$, so the same applies to $\bar{w}^{(1)}\bar{w}^{(2)}\dots $. 
Now, we use the $\dbar$-shadowing property of $X^+$ and we find $x\in X^+ \subseteq X$ such that
\begin{equation}\label{distanceof-x-and-wbar}
    \dbar\left(x,\Bar{w}^{(1)}\Bar{w}^{(2)}\ldots\right)<\eps/3.
\end{equation} It follows from \eqref{distanceof-barw-and-w} and \eqref{distanceof-x-and-wbar} that
$\dbar\left(x,w^{(1)}w^{(2)}\ldots\right)<\eps$, which concludes the proof.
\end{proof}
%



\section{$\dbar$-approachable examples of proximal and minimal shift spaces}
\label{sec:examples}

Before presenting the details of our constructions, we first recall the necessary background.


An (oriented) $\alf$-labeled (multi)graph is a triple $G = (V,E,\tau)$, where $V$ is the (finite) set of vertices, $E\subset V\times V$ is the edge set, and $\tau \colon E \to \alf$ is the label map. For each $e \in E$ we write $i(e), t(e) \in V$, to denote, respectively, the initial vertex and the terminal vertex of $e$. We say that a sequence (finite or infinite) consisting of $\ell\in\N_0\cup\{\infty\}$ edges $e_1, e_2, \dots$ in $E$ is a \emph{path} of length $\ell$ in $G$ if for every $i < \ell$ we have that $t(e_i)=i(e_{i+1})$. A path $e_1, e_2, \dots, e_\ell$ is \emph{closed} if $t(e_\ell)=i(e_1)$.
	
Given an oriented $\alf$-labeled graph $G = (V,E,\tau)$, we define the shift $X_G \subseteq \FS$ by reading off labels of all
infinite paths in $G$. In other words, $X_G$ is the set of all $x \in \FS$ such that $x_i = \tau(e_{i+1})$ for each $i \ge 0$ for some path $e_1,e_2, \dots$ in $G$.  We say that $X$ is a \emph{sofic shift} if there exists a labeled graph $G= (V,E,\tau)$
such that $X$ is presented by $G$, meaning that $X=X_G$.  Every shift of finite type is sofic. A sofic shift is transitive if and only if it can be presented by a (strongly) connected graph
(each pair of vertices can be connected by a path), see \cite[Prop. 3.3.11]{LM}. A sofic shift is topologically mixing if and only if it can be presented by a (strongly) connected aperiodic graph, i.e. the graph with two closed paths of coprime lengths.

To prove the properies of shift spaces resulting from our constructions we will use the following result, which is a direct corollary of a combination of Theorem 6 and Corollary 17 in \cite{KKK2}.

\begin{thm}\label{gen-scheme-summ}
Let $(X_n)_{n=1}^\infty$ be a decreasing sequence of mixing sofic shift spaces over $\alf$ such that
\[
\sum_{n=1}^\infty\Hdbar(X_n,X_{n+1})<\infty,
\]
Then $X=\bigcap_{n=1}^\infty X_n$ is a $\dbar$-approachable and chain mixing shift space such that $\Hdbarm(\Mse(X_n),\Mse(X)) \to 0$ as $n \to \infty$. In particular, $X$ satisfies $\sigma(X)=X$ and has the $\dbar$-shadowing property.
\end{thm}
Let us recall that \cite[Corollary 17]{KKK2} is stated using the lower density. We do not need such flexibility here, so we stated it in terms of the $\dbar$-pseudodistance. 

In fact, we would like to apply Theorem \ref{gen-scheme-summ} to a sequence of mixing sofic shifts that is not decreasing. 
A natural way to apply Theorem \ref{gen-scheme-summ}  is to replace the sequence of sofic shifts $(X_n)_{n=1}^\infty$  with the decreasing sequence of shift spaces $(Y_n)_{n=1}^\infty$, where $Y_n:=X_1\cap\ldots\cap X_n$ for $n\in\N$. It is easy to see that thus defined shift $Y_n$ is also sofic for $n \in \N$. Indeed, if  for $m=1,2,\ldots,n$ a labeled graph $G_m = (V_m,E_m,\tau_m)$ presents the sofic shift space $X_m$, then $Y_n$ is a sofic shift presented by the graph $G=(V,E,\tau)$, where $V=\prod_{1\leq k\leq n}V_k$ and there is an edge from $(v_1,\ldots,v_n)\in V$ to $(v'_1,\ldots,v'_n)\in V$ with label $\ell\in\alf$ if and only if for every $1\leq k\leq n$ in the graph $G_k$ there is an edge from $v_k$ to $v'_k$ labeled with $\ell$. We say that the graph $G=(V,E,\tau)$  is the \emph{coupling} of graphs $G_k$, $1\leq k\leq n$. Unfortunately, $Y_n$ need not be mixing even if the shift spaces $X_1,\ldots, X_n$ are.
The problem is that strong connectedness of the graphs $G_k$ for $1\leq k\leq n$ need not ensure strong connectedness of their coupling $G$. Hence, the induced sofic shift $Y_n$ need not be transitive and its ergodic measures need not be dense. Indeed, Figure \ref{fig:no-safe-symbol} shows two graphs whose coupling is a sofic shift which is not transitive and whose ergodic measures are not dense in the set of all invariant measures. The same situation takes place for the sofic shifts represented in Figure \ref{fig:period}.

\begin{figure}[htbp]
  \centering
  \includegraphics[scale=.75]{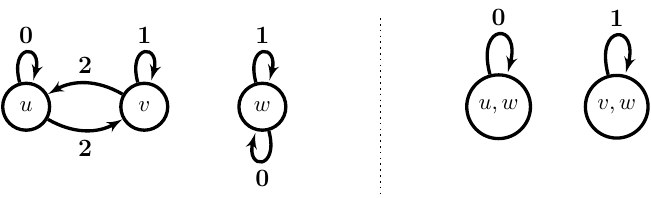}
\caption{Two graphs (left) whose coupling (right) is disconnected. The first graph has no safe symbol.}
\label{fig:no-safe-symbol}
\end{figure}

However, under mild additional assumptions we can ensure that the sofic shifts $Y_n$  are transitive. Let $X$ be a sofic shift over the alphabet $\alf$ and pick a labeled graph $G = (V,E,\tau)$ presenting $X$. We call a symbol $b \in \alf$ a \emph{safe symbol} for $X$ if for every edge $e\in E$  that goes from a vertex $v\in V$ to a vertex $v'\in V$, there is an edge $e' \in E$ from $v$ to $v'$ with label $\tau(e')=b$. The \emph{period} of a graph $G = (V,E)$ is the greatest common divisor of the lengths of all cycles. Every sofic shift also has a period, which is the greatest common divisor of periods of its presentations through labeled graphs. A graph is \emph{aperiodic} if it has period $1$. Every sofic shift with period $1$ has an aperiodic presentation. We will use a standard fact from graph theory, stating that if $G = (V,E)$ is a strongly connected graph with period $m$ then the set of lengths of paths between any pair of vertices $u,v \in V$ is the set-theoretic difference of an infinite arithmetic progression with step $m$ and a finite set.

\begin{figure}[hb]
  \centering
  \includegraphics[scale=.75]{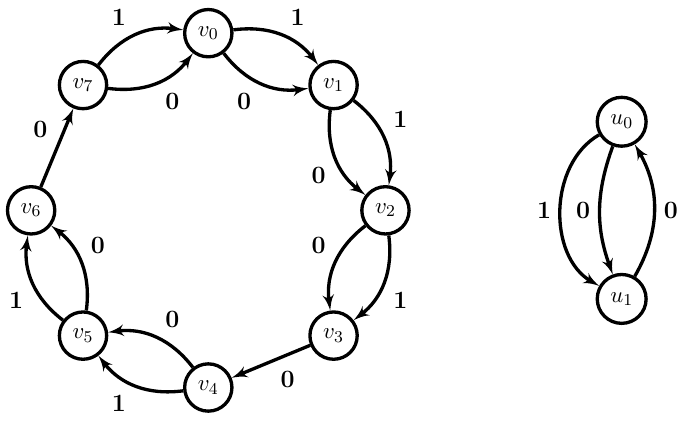}
  \caption{Two graphs whose periods are not coprime (above) and their coupling (below).}
  \includegraphics[scale=.75]{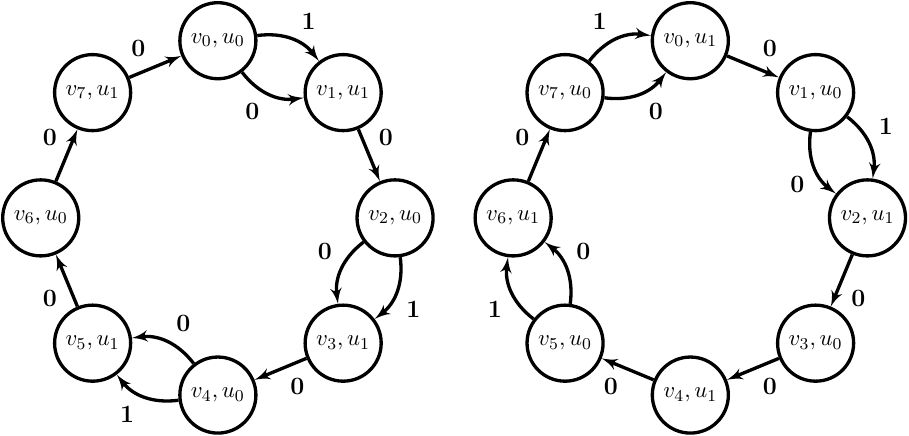}
  \label{fig:period}
\end{figure}

Note that the two shifts in Figure \ref{fig:no-safe-symbol} are aperiodic but lack a common safe symbol, while the two shifts in Figure \ref{fig:period} share a safe symbol $0$ but have positive periods $8$ and $2$. Hence, neither assumption in the following proposition can be removed.

\begin{prop}\label{prop-connected-coupling}
    If $G_k=(V_k,E_k,\tau_k)$ for $1\leq k\leq n$ are strongly connected labeled graphs with a common safe symbol and pairwise coprime periods, then their coupling $G$ is strongly connected.
\end{prop}

\begin{proof} Let $(v_1,v_2,\ldots,v_k)$ and $(v'_1,v'_2,\ldots,v'_n)$ be any two vertices in $G$. For any $1\leq k\leq n$, there exists a walk from $v_k$  to $v'_k$ in the graph $G_k$. Let $\ell_k$ denote the length of one such walk and let $m_k$ denote the period of $G$. For all sufficiently large $j \in \N_0$ there exists a walk of length $jm_k+\ell_k$ from $v_k$  to $v'_k$. Since the graphs under consideration have a common safe symbol $b$, we may additionally assume that the edges in aforementioned paths are all labelled with $b$. Since the integers $m_k$ are coprime, there exists infinitely many integers $\ell$ such that $\ell \equiv \ell_k \bmod{m_k}$ for each $1\leq k \leq n$. Consequently, we can find $\ell_0$ such that for each $1\leq k \leq n$ there exists in $G_k$ a walk from $v_k$ to $v_k'$ of length $\ell_0$, consisting only of edges labeled with $b$. This walk induces a walk length $\ell_0$ from $(v_1,v_2,\ldots,v_k)$ to $(v'_1,v'_2,\ldots,v'_n)$ in $G$.
\end{proof}

\begin{cor}\label{cor-finite-intersection-of-sofic}
Let $n\in \N$. If $X_{k}$ for $1\leq k\leq n$ are transitive sofic shifts with a common safe symbol and pairwise coprime periods, then $Y=X_1\cap\ldots \cap X_n$ is a non-empty transitive sofic shift with a safe symbol. Furthermore, if for each $1\leq k\leq n$ the shift space $X_k$ is topologically mixing, then $Y$ is also topologically mixing.
\end{cor}

\subsection{A $\dbar$-approachable proximal shift space}
\label{proximal}
We are going to construct a $\dbar$-ap\-proachable and topologically mixing proximal shift. Furthermore, our example is hereditary and has positive topological entropy, hence its ergodic measures are ent\-ro\-py-dense and its simplex of invariant measures is the Poulsen simplex. Assume that $\alf\subset\N_0$. A shift space $X\subseteq\FS$  is \emph{hereditary} if for every $x\in X$ and $y\in\FS$ with $y_i \leq x_i$ for all $i \ge 0$ we have $y \in X$. The \emph{hereditary closure} $\tilde X$ of a shift space $X$ is the smallest hereditary shift containing $X$. That is, $\tilde{X}$  consists of all $y \in \FS$ such that there exists $x=(x_i)_{i\ge 0} \in X$ with $y_i \leq x_i$ for all $i \ge 0$. For more on hereditary shifts, see \cite{K}. For some special properties of the simplex of invariant measures of hereditary shifts see Remark \ref{rem:her} below. Recall that a hereditary shift space $X$ is proximal if for every $N>0$ and $x\in X$ the word $0^N$ appears with bounded gaps in $x$ (see the discussion of the Theorem B and the proof of Proposition 3.3 in \cite{DKKPL}).

\begin{exmp}\label{ex:proximal}
For $n\in\N$ we consider a sofic shifts $Z_n$ presented by a labeled graph $G_n = (V_n,E_n,\tau_n)$ where $V_n=\{v_0, v_1, \ldots,v_{10^n-1}\}$ and  edges and their labels given by:
\begin{itemize}
\item for every $0 \leq k < 10^n$, there is an edge from $v_k$ to $v_{k+1}$ with label $0$, where $v_{10^n} = v_{0}$;
\item for every $1\leq k\leq 10^n-2^n$, there is an edge from $v_k$ to $v_{k+1}$ with label $1$,
\item there is an edge from $10^n-2^n$ to $10^n-2^n+2$ with label $0$.
\end{itemize}
Let $Z = \bigcap_{n=1}^\infty Z_n$. Then for each $n \geq 1$ we have $Z_{n+1} \not \subset Z_{n}$ and $Z_n \not \subset Z_{n+1}$.
\end{exmp}

\begin{figure}[htbp]
  \centering
  \includegraphics[scale=.75]{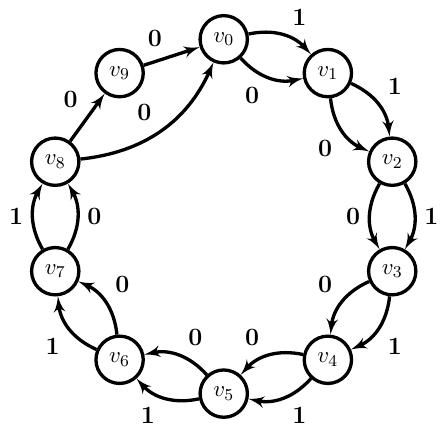}
  \caption{The graph from Example \ref{ex:proximal} with $n=1$.}
  \label{fig:proximal}
\end{figure}

\begin{prop}\label{prop:proximal}
  The shift space $Z$ defined in Example \ref{ex:proximal} is hereditary, topologically mixing, proximal, has positive topological entropy, and the ergodic measures $\Mse(Z)$ are entropy-dense in $\Ms(Z)$.
\end{prop}

\begin{proof}
  Since  all $Z_n$'s are hereditary, $0$ is their common safe symbol and the shift $Z$ is hereditary as well. Furthermore, it contains a sequence where $1$'s appear with positive density, hence $Z$ has positive entropy. For every $k\in\N$, in the graph $G_k$ presenting $Z_k$ in Example \ref{ex:proximal} there are two closed walks of coprime lengths $10^k$ and $10^{k}-1$, whence $G_k$ is aperiodic. By Corollary \ref{cor-finite-intersection-of-sofic}, for each $n \in \N$ the intersection $Y_n:=Z_1\cap\ldots\cap  Z_{n}$ is a topologically mixing sofic shift. In particular, the ergodic measures of $Y_n$ are entropy dense (\cite{EKW}). By Corollary \ref{gen-scheme-summ}, in order to conclude that the ergodic measures on $Z$ are entropy-dense in $\Ms(Z)$, it suffices to check that
\[\sum^\infty_{n=1}\Hdbar(Y_n,Y_{n+1})< \infty.\]
Fix $n\in\N$. To bound $\Hdbar(Y_n,Y_{n+1})$, consider $x\in Y_n$. Define $y\in \{0,1\}^\infty$ by
\[
y_j=
\begin{cases}
  x_j& \text{if } j \bmod{10^{n+1}} \in [0,10^{n+1}-2^{n+1}),\\
  0& \text{if } j \bmod{10^{n+1}} \in [10^{n+1}-2^{n+1},10^{n+1}).
\end{cases}
\]
It is clear that $y\in Z_{n+1}$. Since $Y_n$ is hereditary and $x\in Y_n$, we see that $y$ belongs to $Y_n$ as well. Thus, $y\in Y_n\cap Z_{n+1}=Y_{n+1}$. Furthermore, 
$\dbar(x,y)\leq \left({1}/{5}\right)^{n+1}$, and hence
$\Hdbar(Y_n,Y_{n+1})\leq \left({1}/{5}\right)^{n+1}$.
This completes the proof of entropy density of ergodic measure.

Finally, we prove topological mixing for $Z$. Bearing in mind that $Z$ and $Z_n$ ($n \in \N$) are hereditary, in order to show that $Z$ is topologically mixing it is enough to show that for each $u,v \in \Bl(Z_n)$ there exists $M \in \N$ such that for all $m \geq M$ and all $n \in \N$ we have $u 0^m v \in \Bl(Z_n)$ (hence $u 0^m v \in \Bl(Z)$).

Fix $u,v$, denote $i=|u|$ and $j=|v|$. Let $N \in \N$ be such that $10^n > i+j+2(2^n-2)$, for all $n\ge N$. Fix $n\ge N$. By the pigeonhole principle, for each $m \in \N$ there exists $t \in \N$ such that $[t,t+i) \bmod 10^n \subset [0,10^n-2^n)$ and $[t+i+m,t+i+m+j)\bmod 10^n \subset [0,10^n-2^n)$. By the definition of $Z_n$, starting a path in the corresponding graph from $v_t$ and ending in $v_{t+i+m+j}$, we can read $u0^mv$ along the path, so the word belongs to $\Bl_{i+m+j}(Z_n)$. We have just proved that for all $n\ge N$ and all $m\in\N$, $u0^mv\in\Bl(Z_n)$. It remains to discuss the case when $n\le N$.

Since for each $n \in \N$ the system $Z_n$ is mixing, there exists $M_n \in \N$ such that $u 0^m v \in \Bl(Z_n)$ for all $m \geq M_n$. Put $M$ to be maximum of $M_n$, $n\le N$. Then for $m\ge M$, $u0^mv\in\Bl(Z_n)$ for all $n\le N$. But we have already proved the same conclusion for $n\ge N$ too. This concludes the proof of topological mixing of $Z$.

Since $Z$ is hereditary to prove that it is proximal it is enough to show that for every $N>0$ and $x\in Z$ the word $0^N$ appears with bounded gaps in $z$ (see \cite{DKKPL}). But every $z\in Z$ must belong to $Y_n$ for every $n\ge 1$, so arbitrarily long blocks of $0$'s appear syndetically (ie. with bounded gaps).
\end{proof}

\begin{rem}\label{rem:her}
Since $Z$ is hereditary and proximal, some of the results from \cite{KKK2,K} apply: $Z$ is
distributionally chaotic of type 2 \cite[Thm. 23]{K}, but not of type 1 \cite[Thm. 23]{K} (cf. also \cite{Oprocha}. Moreover, for each
$t > 0$, the set of all ergodic invariant measures on $Z$ with entropy not exceeding $t$ is arcwise connected
with respect to the $\dbar$-metric on the set of all invariant measures \cite[Thm. 6]{KKK2}.
\end{rem}

\subsection{A $\dbar$-approachable minimal shift space}
\label{sec:minimal}


We are going to construct a minimal shift space which is $\dbar$-approximable by a descending sequence of mixing sofic shifts. The sofic shift $X_n$ in the sequence will be generated by a finite code $\cB_n\subset\{0,1\}^+$.  The parameters of the construction are an initial finite non-empty set of words  $\cB_1$ and a sequence of positive integers $(t(n))^\infty_{n=1}$. We assume that $t(n)\geq 2$ for every $n$. We will impose some more conditions on $t(n)$'s and $\cB_1$ later.

Assume we have defined the family of words $\cB_n$ for some $n\ge 1$. Write $k(n)$ for the cardinality of $\cB_n$. Enumerate the elements of $\cB_n$ as $\betn_1,\ldots,\betn_{k(n)}$, and let $\tau(n)=\betn_1\ldots\betn_{k(n)}$ denote their concatenation. Let
$s(n)$ (respectively, $\ell(n)$) be the length of the \emph{shortest} (respectively, the \emph{longest}) word in $\cB_n$.
Words belonging to $\cB_{n+1}$ are constructed as follows: first we concatenate $t(n)$ arbitrarily chosen words from $\cB_n$, then we add the suffix $\tau(n)$:
\[
\cB_{n+1}=\{b_1 b_2 \dots b_{t(n)} \tau(n) : b_i \in\cB_n \text{ for } 1\leq  i \leq t(n)\}.\]


By the construction, $\ell(n)<\tau(n)<s(n+1)$ for every $n \geq 1$ and so $s(n)\nearrow\infty$ as $n\to\infty$. Moreover, every word from $\cB_n$ is a subword of every word from $\cB_{n+1}$. Recursively,
\begin{equation}\label{cond:B_n_subwords}
\text{$u$ is a subword of $v$, for every $u\in\cB_n$, $v\in\cB_m$, $n\leq m$}.
\end{equation}


For $n\ge 1$ let $X_n$ be the coded shift generated by the code $\cB_n$. That is, $X_n$ consists of all concatenations of words from $\cB_n$ together with their shifts. Since $\cB_n$ is finite, the shift $X_n$ is transitive and sofic. It follows from \eqref{cond:B_n_subwords} that $X_{n+1}\subseteq X_n$. Hence $X=\bigcap_{n=1}^\infty X_n$ is a non-empty shift space.


\begin{prop}\label{prop:X-is-minimal}
The shift $X$ constructed above is minimal.
\end{prop}
\begin{proof}
If $|\cB_1|=1$, then $X$ is an orbit of a periodic point. Let us assume that $|\cB_1|>1$, so $|\cB_n|>1$ for every $n$.
  We need to prove that if $u\in\lang(X)=\bigcap_{n=1}^\infty\lang(X_n)$ and $x\in X$, then $u$ is appears in $x$. Fix $x\in X$ and $u\in\lang(X)$. Take $n$ large enough to imply $|u|< s(n)$. Since $\lang(X)\subseteq\lang(X_n)$, we see that $u$ must appear in some $x'\in X_n$. As all words in $\cB_n$ are longer than $u$ and $x'$ is a shift (possibly trivial) of an infinite concatenation of words from $\cB_n$, we conclude that $u$ is a subword of some $\bar u\in\lang(X_n)$ which is the concatenation of two words $v,w$ in $\cB_{n}$ (one of them might be empty). By the definition of $\cB_n$'s, every concatenation $vw$ for words $v$ and $w$ from $\cB_n$ appears in some word from $\cB_{n+1}$, therefore $vw$ and in particular $u$ is a subword of a word $w'\in\cB_{n+1}$. Hence, condition (\ref{cond:B_n_subwords}) ensures that $u$ is a subword of all words from $\cB_{n+2}$. But $x\in X_{n+2}$, so it is a shifted infinite concatenation of words from  $\cB_{n+2}$. In particular, some word from $\cB_{n+2}$ appears in $x$, and so does $u$.
\end{proof}


From now on, we set $\cB_1=\{0,11\}$. For this choice of $\cB_1$ a simple inductive argument shows that for each $n \geq 1$, the set of lengths of all words in $\cB_n$ is an interval:
\begin{equation}\label{cond:diff-1-strong}
\text{for every $m$ with $s(n) \leq m \leq \ell(n)$, there exists $u \in \cB_n$ with $\abs{u} = m$.}
\end{equation}

\begin{prop}For every $n\ge 1$ the coded system $X_n$ is a mixing sofic shift. 
\end{prop}
\begin{proof}
Any coded system generated by a finite sequence of words is sofic. In the corresponding graphs, every word in $\cB_n$ is represented by a cycle, and all these cycles have a  common vertex. In particular, the graph presenting $X_n$ is strongly irreducible and $X_n$ is transitive. In addition, it follows from \eqref{cond:diff-1-strong} that there are two words in $\cB_n$ with co-prime lengths, so the graph is aperiodic, thus $X_n$ is mixing and has the specification property.
\end{proof}

In the rest of the section, it will be convenient to control the ratio $\frac{s(n)}{\ell(n)}$. Because of the identities
\begin{align*}
	s(n+1) &= t(n)s(n) + \abs{\tau(n)},\\
	\ell(n+1) &= t(n)\ell(n) + \abs{\tau(n)},
\end{align*}
we get that the ratio is increasing and
$$\frac{s(n)}{\ell(n)}\ge\frac12,\qquad n\ge 1.$$

On the other, we can ensure that
\begin{equation}\label{eq:2vs3-ratio}
\frac{s(n)}{\ell(n)}< 2/3, \qquad n\ge 1
\end{equation}
by satisfying the equivalent condition (the equivalence follows from the inductive definition of $s(n)$ and $\ell(n)$ mentioned above)
\begin{equation}\label{eq:2vs3-for-t}
t(n)> \frac{|\tau(n)|}{2\ell(n)-3 s(n)}, \qquad n\ge 1.
\end{equation}
Since $|\tau(n)|$, $\ell(n)$ and $s(n)$ are determined by $t(i)$, $1\le i<n$, we have enough freedom to construct the sequence $t(n)$ satisfying condition \eqref{eq:2vs3-for-t} in an inductive way.

\begin{prop}\label{prop:minimal_Cauchy}
  Let $\eps>0$ and $t(n)$ be such a sequence that satisfies condition \eqref{eq:2vs3-for-t} and
\begin{equation}\label{eq:tn-large-for-d}
t(n)> \frac{|\tau(n)|+3\ell(n)}{s(n)\eps 2^{-n}}, \qquad n\ge 1.
\end{equation}
Then
\begin{equation}\label{eq:361}
  \sum^\infty_{n=1}\Hdbar(X_n,X_{n+1})<\eps.
\end{equation}
\end{prop}
\begin{proof}



Put $\eps_n=\eps 2^{-n}$. We will show that $\Hdbar(X_n,X_{n+1}) < \eps_n$ for all $n \geq 1$, which directly implies \eqref{eq:361}. Fix $n\ge 1$ and $y\in X_n$. Our goal is to find $z\in X_{n+1}$ such that $\dbar(y,z)<\eps_n$.
Since $\dbar$ is shift invariant, without loss of generality we assume that $y$ is a concatenation of blocks from $\cB_n$, that is, we have
\[
y=b_1b_2b_3\ldots,\quad\text{where }b_j\in\cB_n\text{ for }j=1,2,\ldots.
\]

We will construct $z$ inductively. First, we note that the word
\[w=b_1b_2b_3\ldots b_{t(n)} \tau(n) \]
belongs to $\cB_{n+1}$. Let $j \geq t(n)$ be the index with
\[
|b_1b_2\ldots b_j| \le |w| < |b_1b_2\ldots b_{j+1}|,
\]
and let $a$ be the suffix of $b_{j+1}$ such that $|b_1b_2\ldots b_j b_{j+1}| = |w| + |a|$.
We observe that there exist words $b_{1}', b_{2}', b_{3}' \in \cB_{n} \cup \{\emptyword\}$ such that $|b_{1}'b_{2}' b_{3}'| = |a b_{j+2} b_{j+3}|$. Indeed, if $2s(n) \leq |a b_{j+2} b_{j+3}| \leq 2\ell(n)$ then it follows from \eqref{cond:diff-1-strong} that we can find $b_{1}', b_{2}' \in \cB_{n}$ and $ b_{3}' =\emptyword$ with the required total length. Likewise, if $3s(n) \leq |a b_{j+2} b_{j+3}| \leq 3\ell(n)$ then we can apply the same argument with $b_{1}', b_{2}', b_{3}' \in \cB_{n}$. Since $|a| < \ell(n)$ and $3s(n) < 2\ell(n)$ (we assumed (\ref{eq:2vs3-for-t}), which is equivalent to (\ref{eq:2vs3-ratio})), these two cases cover all possibilities.

For $i \geq 3$, define $b'_{i} = b_{j+i}$ and let $y' = b'_{1}b'_{2} b'_{3}\dots$. Then $y$ and $w y'$ differ only on the positions where $\tau(n)b'_{1}b'_{2} b'_{3}$ appears in $w y'$, that is, at most on positions between $|w|-|\tau(n)|$ and $|w|+ 3\ell(n)$. Let us point out that $y'$ is again a concatenation of blocks from $\cB_n$, even in the case when $b_3'$ is the empty word. Hence, we can apply the same reasoning to $y'$, $y''$, $y'''$ and so on, to obtain the word $z = w w' w'' \dots$. Here, we adopt the convention that $y'',w'$ are constructed from $y'$ in the same way as $y',w$ were constructed from $y$, and accordingly for further steps in the construction.

Since $w,w',w'' \in \cB_{n+1}$, we have $z \in X_{n+1}$. Note that for each $i \geq 0$ we have $s(n)t(n) + \tau(n) \leq |w^{(i)}| \leq \ell(n)t(n) + \tau(n)$, and consequently
\[
	\dbar(y,z) \leq \limsup_{i \to \infty} \frac{i \tau(n) + 3i\ell(n)}{|w|+|w'| + \dots |w^{(i-1)}|} \leq \frac{ \tau(n) + 3\ell(n)}{\tau(n)+s(n)t(n)} < \eps_n.
      \]
In the last inequality we use the condition (\ref{eq:tn-large-for-d}) that is stronger.

Hence, $z$ has all of the required properties.
\end{proof}

\begin{prop}\label{prop:minimal_also_mixing}
  Let the sequence $t(n)$, $n\ge 1$, satisfy (\ref{eq:2vs3-ratio}) and
\begin{align}
	\label{eq:92:1} t(n) & \geq \frac{\ell(n)}{\ell(n)-s(n)} & \text{ for all $n \geq 1$};\\
	\label{eq:92:2} t(n) & \geq \frac{2s(n)+2\ell(n) + 3\abs{\tau(n)}}{\ell(n)} & \text{ for all $n \geq 1$}.
\end{align}
Then $X$ is mixing.
\end{prop}

\begin{proof}
We need to show that for all $u,v \in \cB(X)$ there exists $M$ such that for each $m \geq M$ there exists a word $w$ with $\abs{w} = m$ such that $uwv \in \cB(X)$. Note that we can freely replace $u,v$ with any other words $u',v' \in \cB(X)$ which contain $u,v$ as subwords. Hence (repeating an argument from the proof of Prop.\ \ref{prop:X-is-minimal}) we may assume that $u,v \in \cB_n$ for some $n \geq 1$. Proceeding by induction on $m$, we will prove the following statement:

\textbf{Claim:} For all $m \geq 0$, for all $n \geq 1$ such that $2s(n) \leq m$, for all $u,v \in \cB(X)$, there exists $w$ with $\abs{w} = m$ and $uwv \in \cB(X)$.

We may assume that the claim above has been proved for all $m' < m$. We consider three cases depending on the magnitude of $m$.

\textit{Case 1:} Suppose first that $2s(n) \leq m \leq (t(n)-2)\ell(n)$. Note that for each $j \geq 2$ it follows from (\ref{eq:2vs3-ratio}) that
\[
	(j+1)s(n) \leq \frac{3}{2} j \cdot \frac{2}{3} \ell(n) = j \ell(n).
\]
As a consequence, the intervals $\left[ js(n), j\ell(n) \right]$ for $2 \leq j \leq t(n)-2$ fully cover the interval $\left[2s(n),(t(n)-2) \ell(n)\right]$. Hence, we can find $j$ with $2 \leq j \leq t(n)-2$ such that $js(n) \leq m \leq j\ell(n)$. It follows from \eqref{cond:diff-1-strong} that there exists a word $w$ with $\abs{w} = m$ of the form $w = b_1b_2\dots b_j$ with $b_1,b_2,\dots,b_j \in \cB_n$. Thus, $uwv$ is a prefix of $\cB_{n+1}$ (specifically, of any word of the form $ub_1b_2\dots b_j v b_1' b_2\dots b_i \tau(n)$, where $i = t(n)-j-2$ and $b_1',b_2',\dots,b_i'\in \cB_n$). It follows that $uwv \in \cB(X)$.

\textit{Case 2:} Suppose second that $(t(n)-2)\ell(n) < m \leq (2t(n)-2)\ell(n)+\abs{\tau(n)}$. Then, by \eqref{eq:92:2} we have $m \geq 2 s(n) + \abs{\tau(n)}$. Arguing similarly as in the first case, we can find a word $w$ with $\abs{w} = m$ of the form
\[
w = b_1b_2\dots b_j \tau(n) b_1'b_2'\dots b_k',
\]
where $1 \leq j,k \leq t(n)-1$ and $b_1,b_2,\dots b_j, b_1',b_2',\dots, b_k' \in \cB_n$. Thus, $uwv$ is a subword of the concatenation $c_1c_2$ of two words $c_1,c_2 \in \cB_{n+1}$. Hence, $uwv$ is a subword of a word in $\cB_{n+2}$ and thus $uwv \in \cB(X)$.

\textit{Case 3:} Suppose third that $m > (t(n)-2)\ell(n) + \tau(n)$. Put
\begin{align*}
	u' &= b_1b_2\dots b_{t(n)-1} u \tau(n) \in \cB_{n+1},\\
	v' &= b_2'b_3' \dots b_{t(n)}' \tau(n) \in \cB_{n+1},
\end{align*}
where $b_1,b_2,\dots,b_{t(n)-1},b_2',b_3',b_{t(n)}' \in \cB_n$ are arbitrary. Put also $n' = n+1$ and $m' = m-\abs{\tau(n)}$. By \eqref{eq:92:1}, we have
\[
	m' \geq 2\left(t(n)s(n)+\abs{\tau(n)}\right) = 2 s(n').
\]
Hence, by the inductive assumption, there exists a word $w'$ with $\abs{w'} = m$ such that $u'w'v' \in \cB(X)$. It remains to notice that $uwv$ is a subword of $u'w'v'$, where $w = \tau(n) w$ has length $\abs{w} = \abs{\tau(n)} + \abs{w'} = m$.
\end{proof}

\begin{thm}\label{thm:exist_minimal_dshadowing}
  There exists a sequence of positive integers $(t(n))_{n=1}^\infty$ such that $X$ is minimal, mixing and has positive entropy and $\dbar$-shadowing property. In particular, the shift space $X$ has entropy-dense and uncountable set of ergodic measures.
\end{thm}

\begin{proof}
  Since $X_1$ is a mixing sofic shift it has topological entropy $h>0$. By the uniform continuity of the entropy function with respect to $\dbarm$-distance, there is $\eps>0$ such that for every $\mu,\mu'\in\Ms(\{0,1\}^\infty)$ with $\dbarm(\mu,\mu')<\eps$ we have   \( |h(\mu)-h(\mu')|<h/3\). Fix this $\eps$.

  All the conditions (\ref{eq:2vs3-for-t}), (\ref{eq:tn-large-for-d}) with respect to $\eps$ specified above, (\ref{eq:92:1}) and (\ref{eq:92:2}) can be satisfied simultaneously by one sequence $t(n)$, $n\ge 1$. Indeed, it is enough to construct the sequence inductively and take $t(n)$ large enough with respect to right-hand sides of the conditions which all depend only on the previously taken $t(i)$, $i<n$. Such a sequence then satisfies all assumptions of Propositions \ref{prop:minimal_Cauchy} and \ref{prop:minimal_also_mixing}. Hence, $X$ is mixing and $\dbar$-approachable from above by mixing sofic shifts. In particular, the shift has entropy-dense set of ergodic measure and is $\dbar$-shadowing.

Now, it suffices to find two invariant measures on $X$ with different entropies. By the variational principle, there is an ergodic invariant measure $\nu_1$ on $X_1$ with entropy $h$. Since $X_1$ contains a periodic point, there is also an ergodic invariant measure $\nu_2$ on $X_1$ of entropy zero. The inequality $\Hdbarm(\Mse(X_n),\Mse(X))<\eps/2$ ensures that there is an invariant measure $\nu'_1$ on $X$ that is $\eps$-close to $\nu_1$ in $\dbarm$-distance and so $h/3$-close to $\nu_1$ in entropy. By the same argument, there is an invariant measure $\nu'_2$ on $X$ that is $h/3$-close to $\nu_2$ in entropy. Since the difference between $h(\nu_1)$ and $h(\nu_2)$ equals $h$, the measures $\nu'_1$ and $\nu'_2$ have different entropy and in particular are distinct.
  \end{proof}


\section*{Acknowledgments}
The research cooperation between Michal Kupsa and Dominik Kwietniak was partially funded by the program Excellence Initiative --- Research University under the Strategic Programme Excellence Initiative at Jagiellonian University in Kraków. The program supported the research stay of Michal Kupsa in Kraków in May--June 2023 during which the present paper was (almost) finished.
The research of Melih Emin Can  is part of the project No.\ 2021/43/P/ST1/02885 co-funded by the National Science Centre and the European Union's Horizon 2020 research and innovation programme under the Marie Sklodowska-Curie grant agreement no. 945339.
When this paper was being finished, D.~Kwietniak was partially supported by the Flagship Project "Central European Mathematical Research Lab" under the Strategic Programme Excellence Initiative at Jagiellonian University. J.~Konieczny is supported by UKRI Fellowship EP/X033813/1, and during a significant part of the work on this project he was working within the framework of the LABEX MILYON (ANR-10-LABX-0070) of Universit\'{e} de Lyon, within the program ``Investissements d'Avenir'' (ANR-11-IDEX-0007) operated by the French National Research Agency (ANR). He also acknowledges support from the Foundation for Polish Science (FNP). The authors would like to express their gratitude to Tim Austin for sharing a very early version of his paper \cite{Austin} and an enlightening discussion. We would like to thank Piotr Oprocha for allowing us to use his ideas that lead to the construction of the minimal examples of $\dbar$-approachable shift spaces in Section \ref{sec:minimal}. Thanks are also due to Alexandre Trilles, who helped us to find and fix numerous minor faults in Sections \ref{proximal} and \ref{sec:minimal}.  Melih Emin Can would like to thank Damla Bulda\u{g} Can for her constant love and support. Finally, we are grateful to the referee for careful reading of our paper and for helpful corrections and suggestions.

\end{document}